\documentclass{amsart}
\usepackage{amssymb}

\usepackage{amsfonts}
\usepackage[english]{babel}

\usepackage[T1]{fontenc}
\usepackage{latexsym}

\usepackage{color}
\usepackage[normalem]{ulem}

\newtheorem{theorem}{Theorem}[section]
\newtheorem{lemma}[theorem]{Lemma}
\newtheorem{proposition}[theorem]{Proposition}
\newtheorem{corollary}[theorem]{Corollary}
\newtheorem{definition}[theorem]{Definition}
\newtheorem{remark}[theorem]{Remark}
\newtheorem{example}[theorem]{Example}

\author{Jacek
Jakubowski}
\address{
Institute of Mathematics, University
of Warsaw \\
  Banacha 2, 02-097 Warszawa, Poland \\
e-mail: {\tt jakub@mimuw.edu.pl } }
\author{ Maciej Wi\'sniewolski}
\address{Institute of Mathematics, University of Warsaw,  Banacha 2 \\
 02-097 Warszawa, Poland;\\ e-mail: {\tt  wisniewolski@mimuw.edu.pl }  }

\title[On matching diffusions, Laplace transforms and partial differential equations]
 {On matching diffusions, Laplace transforms and partial differential equations} 

\begin{document}

\maketitle

\begin{center}
First version  March 9, 2014.
\end{center}
\begin{center}
 This version \today.
\end{center}

\begin{abstract}
We present the idea of intertwining of two diffusions by  Feynman-Kac operators.
We present some variations and implications of the method and
give examples of its applications. Among others, it turns out to be a very useful tool for finding the expectations of some functionals of diffusions, especially for computing the Laplace transforms  of stochastic processes.
The examples give new results on marginal distributions of many stochastic processes including a generalized squared Bessel processes and joint distribution for squared Bessel bridge and its integral - the close formulae of the Laplace transforms are presented. We finally present a general version of the method and its applications to PDE of the second order. A new dependence between diffusions and solutions of hyperbolic partial differential equations is presented. In particular, the version of Feynman-Kac representation for hyperbolic PDE is given.  It is presented, among others, the simple form of Laplace transform of wave equation with axial symmetry.
\end{abstract}

\noindent
\begin{quote}


\end{quote}

\section{Introduction} The idea presented in this paper concerns a new look on the problem of intertwining of two Markov processes. For an introduction to this topic see Carmona, Petit and Yor \cite{CPY}, Biane \cite{BI} or Holley and Stroock \cite{HS}. The methods in \cite{BI} and \cite{CPY} are based on  finding  a Markov kernel $\Lambda$ such that $P_t\Lambda = \Lambda \hat{P}_t$, where $P_t, \hat{P}_t$ are two given Markov semi-groups. In this paper instead of finding a Markov kernel $\Lambda$, we place the problem on building a relation between two Feynman-Kac operators. We find the diffusion $\hat{X}$ with generator $\hat{A}$ and, possibly on a different probability space, a diffusion $X$ and generator $A$ such that for some function $V$ we have $P_t\hat{A}f = AP_tf - Vf$.
This results in obtaining the new probabilistic representations of solutions of hyperbolic PDE's. It also  provides a new technique of computing the expectations of some functionals of diffusions, in particular, it enables to calculate the Laplace transforms for a broad class of diffusions.
We can implement this idea, due to the fact
that to a certain diffusion  we can assign another diffusion with
 the known distribution and for these two diffusions the
expectations of properly chosen  functionals are equal.
This procedure is called here the matching diffusions technique and it is in fact another look on the intertwining of diffusions. In our previous work \cite{jw} some elements of this idea has appeared in a context of
  hyperbolic Bessel processes, however, we have discovered that it can be extended to the efficient computational method.

In  Section 2 we present theorem  which allows to obtain the closed
forms of Laplace transforms for many diffusions (Theorem \ref{tw:Lap}). It gives the new results as well as the short proofs of known results, for example that of Faraud and Goutte \cite{FG} (see Corollary \ref{wn-FG}).
In this section we compute the Laplace transforms
of generalized squared Bessel processes.  We generalize the results of Pitman and Yor
\cite{PY} and Donati-Martin and Yor \cite{DYM} where results were obtained by
using the additive property (see also Chapter XI in \cite{RY} to review the standard methodology for squared Bessel processes and Chapter 6 in \cite{JYC}). Our way of finding the Laplace transform of generalized squared Bessel processes relies on finding an unique solution of some linear PDE of the first order. After proving uniqueness of that solution,  we effectively obtain close formulae for  solutions using the characteristic method. This works well for a generalized Bessel processes including an radial Ornstein-Uhlenbeck process.
For a generalized Bessel processes for which the matching diffusion method fails to work since the assumption of belonging to   $\mathcal{FK}$-class is violating, we present modification allowing to obtain the required  results.

In  Section 3 we extend the matching diffusion technique to find $\mathbb{E}e^{-\lambda f(\hat{X}_t)}$ for some function $f$.
We apply that to a diffusion $\hat{X}$ with coefficients satisfying some simple relation to obtain that  we can assign a function $f$ such that $f(\hat{X})$ is associated to a geometric Brownian motion.
This new result corresponds to Zvonikin's observation and Doss-Sussmann method, for details see   \cite[Section 27 Chapter 5]{RW}.
Next, the method of finding the joint Laplace transform of vector $(f(\hat{X}_t),\int_0^tf(\hat{X}_u)du)$ for certain functions $f$ and the general squared Bessel processes $\hat{X}$ is developed.
Using this general results we  find close formula for a squared Bessel bridge and a radial Ornstein-Uhlenbeck process. All these results are new (compare with Laplace transforms of some Brownian functionals given in \cite{JYC} and \cite{BS}).

In Section 4 we present  the matching diffusion technique for a general function $f$.
It results in new probabilistic representation
for solutions of the extended class of second order PDE's.
 A version of Feynman-Kac representation for PDE of hyperbolic type is the crucial result presented here (Theorem \ref{FKhyp}). It gives a new light for connections between some diffusions and bounded, or of polynomial growth, solutions of partial differential equations of the second order. Example \ref{ex4.7} illustrates the efficiency of this theorem.
As another example illustrating the use of the matching diffusion technique we present the method of finding the  simple form of Laplace transform of a solution of the wave equation with axial symmetry (compare to \cite{P} and \cite{PZ}).
At the end, using  the matching diffusion technique, we obtain some new results on Jacobi diffusions for which
  only few closed formulae for functionals have been obtained so far (see \cite{HK} and \cite{BSII}).

\section{Matching Diffusion for Laplace transform}

All filtrations considered in this paper satisfy usual conditions.
By $B$ and $W$ we usually denote the standard Brownian motions possibly on different probability spaces. 
In the first step we consider a classical diffusion $X$ with coefficients $(\sigma,\mu)$, i.e.
 $X$ is a solution to SDE
\begin{equation} \label{eq:SDE}
d X_u = \sigma(u,X_u)dW_u + \mu(u,X_u)du,  \quad X_0=x,
 \end{equation}
and has the  generator of the form
\begin{align*}A_t = \frac12 \sigma(t,\cdot)\frac{\partial^2}{\partial x^2} + \mu(t,\cdot)\frac{\partial}{\partial x}.\end{align*}
\begin{definition}
For given functions $V$, $g$ and $h$ we say that
a pair $(\sigma,\mu)$ is in a Feynman-Kac class (abrr. $\mathcal{FK}$-class) if SDE \eqref{eq:SDE} has,  for all $x$, a weak solution
 $X$ unique in law, and  the Cauchy problem
\begin{align*}
	\frac{\partial f}{\partial t} = A_tf - Vf + g, \quad f(0,x) = h(x),
\end{align*}
has a unique classical solution $f\in\mathcal{C}^{1,2}$ with the probabilistic representation
\begin{equation*} \label{repr} f(t,x) = \mathbb{E}h(X_t)e^{-\int_0^tV(t-u,X_u)du} +\mathbb{E}\int_0^tg(t-s, X_s)e^{- \int_0^sV(t-u,  X_u)du}ds.
\end{equation*}
\end{definition}
The main technical tool to check whether $(\sigma,\mu)$ is in $\mathcal{FK}$-class is the Feynman-Kac theorem (for details see e.g. \cite[Section 5.7]{KS} and Remark \ref{FKcons}).
In the main theorem of this section we present that the Laplace transform of a diffusion $\hat{X}$ can be written
in terms of another well matching diffusion. This approach allows to obtain the closed
forms of Laplace transform for many diffusions.
\begin{theorem} \label{tw:Lap}
  Let $(\hat{X},B)$ and $(X,W)$ be two unique weak solutions, possibly on two different probability spaces $(\Omega,\mathcal{F},\mathbb{P})$ and $(\widetilde{\Omega}, \widetilde{\mathcal{F}}, \widetilde{\mathbb{P}})$, to the following SDEs
\begin{align} \label{tjX-hat}
	\hat{X}_t &= x + \int_0^t\hat{\sigma}(u,\hat{X_u})dB_u + \int_0^t\hat{\mu}(u,\hat{X}_u)du\\
\label{tjX}
	X_t &= \lambda + \int_0^t\sigma(u,X_u)dW_u + \int_0^t\mu(u,X_u)du,
\end{align}
where $x\geq 0, \lambda\geq 0$.
Assume that $\hat{X}, X$  are nonnegative and for all $t\geq0$ we have
\begin{align}\label{martprop}
	\mathbb{E}\int_0^t\hat{\sigma}^2(u,\hat{X_u})du < \infty, \quad \mathbb{E}\int_0^t|\hat{\mu}(u,\hat{X_u})|du < \infty.
\end{align}
Let $g:[0,\infty)\times[0,\infty)\rightarrow\mathbb{R}$ and $ V:[0,\infty)\times[0,\infty)\rightarrow[0,\infty)$ be  continuous functions. \\
Let $h(\lambda)= e^{-x\lambda}, \lambda \geq 0$. Assume that, for such functions $g, V, h$, the pair $(\sigma,\mu)$
is in $\mathcal{FK}$-class and that for any $t\geq 0$ and $\lambda \geq0$
\begin{equation}\label{con1}
\begin{split}
 & \mathbb{E}e^{-\lambda \hat{X}_t}\Big[\frac{\lambda^2}{2}\hat{\sigma}^2(t,\hat{X}_t) -\lambda\hat{\mu}(t,\hat{X}_t) \Big] \\ & = \mathbb{E}e^{-\lambda \hat{X}_t}\Big[\hat{X}_t^2\frac{\sigma^2(t,\lambda)}{2}-\hat{X}_t\mu(t,\lambda) - V(t,\lambda)\Big] + g(t,\lambda).
\end{split}
\end{equation}
Then for all $t\geq0$ and $\lambda \geq 0$
\begin{equation}\label{con2}
	\mathbb{E}e^{-\lambda \hat{X}_t} = \widetilde{\mathbb{E}} e^{-x X_t - \int_0^t V(t-u, X_u)du} + \widetilde{\mathbb{E}}\int_0^tg(t-s, X_s)e^{- \int_0^sV(t-u,  X_u)du}ds ,
\end{equation}
where $\widetilde{\mathbb{E}}$ corresponds to $\widetilde{\mathbb{P}}.$
\end{theorem}
\begin{proof} Define $p(t,\lambda) := \widetilde{\mathbb{E}}e^{-\lambda \hat{X}_t}$. From It\^o's lemma we have
\begin{equation}\label{eq:ItoLap}
	de^{-\lambda \hat{X}_t} = -\lambda e^{-\lambda \hat{X}_t}\Big(\hat{\sigma}(t,\hat{X_t})dB_t + \hat{\mu}(t,\hat{X}_t)dt\Big) +\frac{\lambda^2}{2}e^{-\lambda \hat{X}_t}\hat{\sigma}^2(t,\hat{X_t})dt.
\end{equation}
 By  (\ref{martprop}), the local martingale in (\ref{eq:ItoLap}) is a true martingale, so taking expectation we obtain
\begin{equation*}
	\mathbb{E}e^{-\lambda \hat{X}_t} = e^{-\lambda x} + \mathbb{E}\int_0^te^{-\lambda \hat{X}_u}\Big(\frac{\lambda^2}{2}\hat{\sigma}^2(u,\hat{X_u})-\lambda\hat{\mu}(u,\hat{X}_u)\Big)du.
\end{equation*}
Now, using (\ref{con1}) and the observation that $p\in C^{1,2}((0,\infty)\times(0,\infty))$ we infer that
\begin{align*}
	\frac{\partial p}{\partial t} = \mathcal{A}_t p - Vp + g, \quad p(0,\lambda) = e^{-x\lambda},	
\end{align*}
where $\mathcal{A}_t$ is the generator of the diffusion $X$. Since the pair $(\sigma,\mu)$ is in $\mathcal{FK}$-class  we have
\begin{equation*}
	p(t,\lambda) =  \widetilde{\mathbb{E}} e^{-x X_t - \int_0^t V(t-u, X_u)du} + \widetilde{\mathbb{E}}\int_0^tg(t-s, X_s)e^{- \int_0^sV(t-u,  X_u)du}ds
\end{equation*}
and the result follows from the definition of the function $p$.
\end{proof}

In what follows,
we abuse notation using $\mathbb{P}$ and $\mathbb{E}$ instead of $\widetilde{\mathbb{P}}$ and $\widetilde{\mathbb{E}}$, since it is clear from context which probability measure
is used.
\begin{remark} \label{FKcons} Belonging of $(\sigma,\mu)$ to $\mathcal{FK}$-class is the crucial point in the proof of Theorem \ref{tw:Lap}. Usually, it is checked by application of the Feynman-Kac theorem.
In particular, to have an appropriate set of assumptions ensuring that a version of Theorem \ref{tw:Lap} for terminal condition holds, we can follow Karatzas and Shreve (see section 5.7 in \cite{KS}).
The conditions given in \cite{KS} are not necessary to have the Feynman-Kac representation - they can be relaxed.
If the coefficients $\mu, \sigma$ do not depend on time, then the change from terminal to initial condition in the Cauchy problem is
immediate. If the coefficients depend on time,
 the generator of diffusion, inducing by reverting of time
 in the Cauchy problem, changes. Thus some additional assumptions have to be added, see e.g. the discussion in \cite[p. 139-150]{F} or \cite{PE} for the case of Levy processes.
\end{remark}
In the next part of  paper we present extensions and applications of Theorem \ref{tw:Lap}. From now, we
assume that $g\equiv 0$. To illustrate how the matching diffusion method works, we give, as a simple corollary, the short proof of
result obtained in \cite[Proposition 2.4]{jw}.
\begin{corollary}\label{HBP} \cite{jw}
For any $\lambda > 0$ and $t\geq 0$ we have
\begin{equation*}
	\mathbb{E}e^{-\lambda \cosh(B_t)} = \mathbb{E}e^{-\lambda e^{B_t} - \frac{\lambda^2}{2}\int_0^te^{2B_u}du}.
\end{equation*}
\end{corollary}
\begin{proof} Let us take $\hat{X}_t = \cosh{B_t}$ and $X_t = \lambda e^{B_t}$. Both are
 diffusions satisfying assumptions of Theorem \ref{tw:Lap} with
\begin{align*}
	\hat{\sigma}^2(z) &= z^2 - 1, \ \hat{\mu}(z) = \frac12 z, \ \hat{X}_0 = 1,\\
	\sigma^2(z) &= z^2, \ \mu(z) = \frac12 z, \ X_0 = \lambda.
\end{align*}
Taking $V(t,\lambda) := \frac{\lambda^2}{2}$ we can easily check that conditions \eqref{martprop} and (\ref{con1})   hold. Using  the Feynman-Kac theorem, we see that  $(\sigma,\mu)$ is in $\mathcal{FK}$-class.
Thus the result follows from Theorem \ref{tw:Lap}.
\end{proof}
In what follows, we consider Laplace transform of  generalized squared Bessel processes II. It
is the class of processes consisting of  unique nonnegative solutions
of the SDE's of the form
\begin{equation}\label{GBPII}
  	\hat{X}_t = x + 2\int_0^t\sqrt{\hat{X_u}}dB_u + \int_0^t(\Delta(u) + 2\theta(u)\hat{X_u})du,
\end{equation}
where  $x\geq 0$, $\Delta$ is a nonnegative  continuous  function and $\theta$ is a continuous function (see \cite{FG}).
This notion generalizes the class of processes defined in \cite{DYM} with $\theta\equiv 0$, hence the name.
It  is the class of processes for which the matching diffusions technique can be applied. Before giving the form of Laplace transform of a general squared Bessel process   we  present some auxiliary  results giving iff conditions for belonging to $\mathcal{FK}$-class,
under the assumption that $\sigma\equiv 0$.
\begin{proposition}\label{CCPI}
Let  $\sigma\equiv 0$ and $\mu:[0,\infty)\times [0,\infty)\rightarrow \mathbb{R}$, $V:[0,\infty)\times [0,\infty)\rightarrow [0,\infty)$ be  continuous functions. Suppose that for any
$x\geq 0$ there exists  a positive unique solution on $[0,\infty)\times [0,\infty)$
of the Cauchy problem
\begin{equation}\label{CPI}
	\frac{\partial p}{\partial t} = \mu(t,\lambda)\frac{\partial p}{\partial \lambda} -V(t,\lambda) p, \quad  p(0,\lambda) = e^{-x\lambda}.
\end{equation}
Then the pair $(\sigma, \mu)$ is
in $\mathcal{FK}$-class if and only if for all $t,\lambda\geq 0$
\begin{align}\label{techcon1}
\mu(t,y(t,\lambda)) &= \mu(t,\lambda)\frac{\partial y (t,\lambda)}{\partial \lambda}, \\
\label{techcon1b} V(t,\lambda) &= \nabla_{t,\lambda}\Big(\int_0^tV(t-u,y(u,\lambda))du\Big),
 \end{align}
 where
\begin{align}\label{techcon2}
y(t,\lambda) &= \lambda + \int_0^t\mu(u,y(u,\lambda))du, \quad  \nabla_{t,\lambda} = \frac{\partial}{\partial t} - \mu(t,\lambda)\frac{\partial}{\partial \lambda}.
\end{align}
\end{proposition}
\begin{proof}
Necessity.
 As $(0, \mu)$ is
in $\mathcal{FK}$-class,
there exists a unique
 $X$ satisfying
\begin{equation} \label{ja1}
X_t = \lambda + \int_0^t\mu(u,X_u)du
\end{equation}
and
\begin{equation} \label{ja2}
p(t,\lambda) = e^{-xX_t -\int_0^t V(t-u,X_u)du }
\end{equation}
 is the unique solution of (\ref{CPI}).
 Hence, defining  $y(t,\lambda) = X_t$, we obtain
\begin{align*}
	-&\frac{\partial}{\partial t}\Big(xy(t,\lambda) + \int_0^tV(t-u,y(u,\lambda))du \Big)\\ &= -\mu(t,\lambda)\frac{\partial}{\partial \lambda}\Big(xy(t,\lambda) + \int_0^tV(t-u,y(u,\lambda))du \Big) - V(t,\lambda)
\end{align*}
Since this equality holds for all $x$ and  $\frac{\partial }{\partial t}y(t,\lambda) = \mu(t,y(t,\lambda))$, by \eqref{ja1},  we obtain \eqref{techcon1} and \eqref{techcon1b}.
Sufficiency. Let $p$ be given by
$p(t,\lambda) = e^{-xy(t,\lambda) -\int_0^tV(t-u,y(u,\lambda))du}$,
where $y$ solves the equation $y(t,\lambda) = \lambda + \int_0^t\mu(u,y(u,\lambda))du$. Conditions (\ref{techcon1})  and \eqref{techcon1b} imply that $p$ is a solution of (\ref{CPI}). From assumption it is unique. Hence $(\sigma, \mu)$, where $\sigma\equiv 0$, is
in $\mathcal{FK}$-class.
\end{proof}
We need also the following  lemma on solutions of  the Cauchy problem for $\sigma\equiv 0$.
\begin{lemma}\label{unique} Let $a,b,d$ be continuous functions on $[0,\infty)$ and $a,b$ be non-positive and such that $a^2(t) + b^2(t) >0$  for every $t\geq 0$. Let $V$ be a continuous nonnegative function on $[0,\infty)\times[0,\infty)$. There exists at most one solution of the following
 problem on $[0,\infty)\times[0,\infty)$
\begin{equation}\label{cauchyy}
	\frac{\partial p}{\partial t} = \Big(d(t) + a(t)\lambda^2+b(t)\lambda\Big)\frac{\partial p}{\partial \lambda}- V(t,\lambda) p, \quad  p(0,\lambda) = h(\lambda),
\end{equation}
where $h$ is a nonnegative continuous function.
\end{lemma}
\begin{proof} Let $t\leq T$. Assume that $p,q$ are solutions of (\ref{cauchyy}) and define $w = (p-q)^2$.
We have $w(0,\lambda)= 0$ for every $\lambda$, and
\begin{align*} \frac{\partial w}{\partial t}
&= \Big(d(t) + a(t)\lambda^2+b(t)\lambda\Big)\frac{\partial w}{\partial \lambda}- 2V(t,\lambda)w \\
&\leq \Big(d(t) + a(t)\lambda^2+b(t)\lambda\Big)\frac{\partial w}{\partial \lambda}.
\end{align*}
Let $u(t,\lambda) = \int_0^{\lambda}w(t,z)dz$.
The last inequality implies, for $\lambda \geq 0$, that
\begin{align*}
\frac{\partial u}{\partial t} &\leq \Big(d(t) + a(t)\lambda^2+b(t)\lambda\Big)w - \int_0^{\lambda}(2a(t)z + b(t))w(t,z)dz\\
&\leq \Big(d(t) + a(t)\lambda^2+b(t)\lambda\Big)w - (2a(t)\lambda + b(t))u,
\end{align*}
where we used integration by parts. Choose $\lambda_0$ such that for $t\leq T$ and $\lambda>\lambda_0$
it holds
$\Big(d(t) + a(t)\lambda^2+b(t)\lambda\Big)\leq 0$.
Thus, for any $\lambda>\lambda_0$,
\begin{align*}
\frac{\partial u}{\partial t} &\leq - (2a(t)\lambda + b(t))u.
\end{align*}
From Gronwall's lemma we obtain $u(t,\lambda) = 0$.  The assertion follows, as $T$ is arbitrary.
\end{proof}

The important example of the generalized squared Bessel process is
  a $\delta$-dimensional ($\delta \geq 0$) squared radial Ornstein-Uhlenbeck process with parameter $\alpha\in\mathbb{R}$, which is  the unique solution of the SDE
\begin{equation}\label{ROU}
	\hat{X}_t = x + 2\int_0^t\sqrt{\hat{X}_s}dB_s + \int_0^t(\delta -2\alpha \hat{X}_s)ds.
\end{equation}
As an application of Theorem \ref{tw:Lap} we find the  Laplace transform of squared radial Ornstein-Uhlenbeck process. The transition density function of radial Ornstein-Uhlenbeck process is given in Salminen and Borodin   \cite{BS} (see Apendix 1 Section 26), however the form of the Laplace transform hasn't been presented there.
\begin{proposition}(Squared Radial Ornstein-Uhlenbeck) \label{SROU} The Laplace transform of $\delta$-dimensional squared radial Ornstein-Uhlenbeck with parameter $\alpha\geq 0$ 
has the form
\begin{equation*}
	\mathbb{E}e^{-\lambda \hat{X}_t} = \exp\Big\{ -\frac{\lambda\delta}{2} - \frac{\alpha(x- \delta/2)}{(1+\alpha/\lambda)e^{2t\alpha}-1}\Big\}.
\end{equation*}
\end{proposition}
\begin{proof}
Let us take
$$
\sigma\equiv 0,\  \mu(t,\lambda) = -2\lambda\Big(\lambda + \alpha\Big),\  V(t,\lambda) = \delta\lambda.
$$
Lemma \ref{unique} and Proposition \ref{CCPI} for $p(t,\lambda) = \mathbb{E}e^{-\lambda \hat{X}_t}$, guarantee that $(0, \mu)$ belongs to $\mathcal{FK}$-class, so we can apply Theorem \ref{tw:Lap}.
The process $X$ is  deterministic and satisfies the ODE
\begin{equation}
	y' = -2y(y+\alpha), \ y(0) = \lambda.
\end{equation}
After solving the above ODE we get
$$
y(t) = X_t = \frac{\alpha}{(1+\alpha/\lambda)e^{2t\alpha}-1}.
$$
We insert the solution $y(t)$ in (\ref{con2}) and some elementary algebra concludes the proof.
\end{proof}
Now we give a short proof of result of Faraud and Goutte \cite{FG}.
\begin{corollary} \cite{FG} \label{wn-FG}
\label{GBPP}(Generalized squared Bessel process) Let $\Delta:[0,\infty)\rightarrow [0,\infty)$ be a continuous function and $x\geq 0$. Suppose that $\hat{X}$ is the unique nonnegative solution of  SDE
\begin{equation}\label{GBP}
   d	\hat{X}_t =  2 \sqrt{\hat{X_t}}dB_t + \Delta(t)dt, \quad \hat{X}_0 = x.
\end{equation}
Then, for $\lambda > 0$,
\begin{equation}\label{LapTGBI}
	\mathbb{E}e^{-\lambda \hat{X}_t} = \exp{\Big\{-\frac{x\lambda}{1+2\lambda t} - \lambda\int_0^t\frac{\Delta(t-u)}{1+2\lambda u}du\Big\}}.
\end{equation}
 \end{corollary}
\begin{proof} We have
\begin{align*}
	\hat{\sigma}^2(z) &= 4z, \ \hat{\mu}(t) = \Delta(t), \ \hat{X}_0 = x\geq 0.
\end{align*}
Using properties of squared Bessel processes we obtain \eqref{martprop}.
Taking
\begin{align*}
	\sigma^2(z) \equiv 0, \ \mu(z) = -2 z^2, \ X_0 = \lambda\geq 0, \ V(t,z) = \Delta(t) z ,
\end{align*}
we see that \eqref{con1} is satisfied.
Taking  $p(t,\lambda) = \mathbb{E}e^{-\lambda \hat{X}_t}$ and using Lemma \ref{unique} and Proposition \ref{CCPI}  we see that  the pair $(0, \mu)$ is in $\mathcal{FK}$. Hence, applying Theorem \ref{tw:Lap} for $\hat{X}$ and $X$ given by
\begin{equation*}
X_t = \frac{\lambda}{1+2\lambda t}
\end{equation*}
we obtain (\ref{LapTGBI}).
\end{proof}
Now, we consider the Laplace transform of $\hat{X}$, a generalized squared Bessel process II.
Observe that equality \eqref{con1} is satisfied for $\hat{X}$,
$p(t,\lambda) = \mathbb{E}e^{-\lambda\hat{X}_t}$ and
\begin{align*}
	\sigma(t,z) \equiv 0, \ \mu(t,z) = -2 z^2+ 2z\theta(t), \ X_0 = \lambda\geq 0, \ V(t,z) = \Delta(t) z.
\end{align*}
So, the matching diffusion procedure can be used, provided
the pair $(0, \mu)$ is in $\mathcal{FK}$-class.
By Lemma \ref{unique}, there exists  a unique solution to the Cauchy problem for $p(t,\lambda) = \mathbb{E}e^{-\lambda \hat{X}_t}$, the Laplace transform of $\hat{X}$, with a non-positive function $\theta$.
Moreover
\begin{lemma} \label{GBPPII}
Let $\Delta:[0,\infty)\rightarrow [0,\infty)$ be a continuous function and  $\theta:[0,\infty)\rightarrow (-\infty,0]$ be a function in $\mathcal{C}^1$. Let $\sigma, \mu$ be given by $\sigma(t,z) \equiv 0, \ \mu(t,z) = -2 z^2+ 2z\theta(t)$.
Then the pair $(\sigma, \mu)$ is
in $\mathcal{FK}$-class if and only if $\theta\equiv c$ for a non-positive constant and
\begin{align} \label{DD}
	 \Delta (t) = c\nabla_{t,\lambda}\Big(\int_0^t\Delta(t-u)\frac{e^{2cu}}{c+\lambda(e^{2cu}-1)}du\Big),
\end{align}
where $\nabla_{t,\lambda}$ is defined in (\ref{techcon2}).
\end{lemma}
\begin{proof}
We use Proposition \ref{CCPI} and prove that  conditions (\ref{techcon1}) and (\ref{techcon1b}) hold. First, observe that if $y$ is given by  (\ref{techcon2}), i.e. $y(t,\lambda) = \lambda + \int_0^t\mu(u,y(u,\lambda))du$, then the solution, as the solution of Bernoulli equation, is of the form
\begin{align*}
	y(t,\lambda) = \frac{\lambda k(t)}{1+2\lambda\int_0^tk(u)du}, \quad  k(t) = e^{2\int_0^t\theta(u)du}.
\end{align*}
For such  $y$ and $\mu$ condition  (\ref{techcon1}), i.e. $\mu(t,y(t,\lambda)) = \mu(t,\lambda)\frac{\partial y (t,\lambda)}{\partial \lambda}$, implies that
\begin{equation} \label{techcon3}
y(t,\lambda)-\theta(t) = \frac{y(t,\lambda)}{k(t)}(1-\theta(t)/\lambda).
\end{equation}
The last in turn leads to
$$k(t) - 2\theta(t)\int_0^tk(u)du = 1,$$
which after taking derivative gives $\theta'\equiv 0$. Condition (\ref{DD}) is precisely (\ref{techcon1b}) for $\theta\equiv c$. This finishes the proof of necessity. To prove sufficiency it is enough to observe that for $\theta\equiv c$ equation (\ref{techcon3}) is satisfied, and (\ref{techcon3}) implies (\ref{techcon1}).
\end{proof}
\begin{corollary} \label{jj31} If $\hat{X}$ is a generalized squared Bessel process II, $\hat{X}_0=x\geq 0$, $\theta\equiv c\leq 0$ and (\ref{DD}) holds, then for $\lambda\geq 0$
\begin{align*}
	\mathbb{E}e^{-\lambda\hat{X}_t} = e^{-xy(t,\lambda) - \int_0^t\Delta(t-u)y(u,\lambda)du},
	\end{align*}
where
\begin{equation*}
y(t,\lambda) = \frac{\lambda ce^{2ct}}{c+2\lambda(e^{2ct}-1)}.
\end{equation*}
\end{corollary}
\begin{proof} Follows from Lemma \ref{GBPPII} and Theorem \ref{tw:Lap}.
\end{proof}

\begin{remark} \label{uw2.8}
We consider separately generalized squared Bessel processes and generalized squared Bessel processes II
for historical reason. The first generalization was considered by Donati-Martin and Yor \cite{DYM}, and the second one by  Faraud and Goutte \cite{FG}.
 In Lemma 3.1  \cite[Section 3]{DYM} Donati-Martin and Yor presented
the conection of functional of  generalized squared Bessel process with a solution of Sturm-Liouville equation. For a measurable function $f:[0,\infty)\rightarrow[0,\infty)$ with compact support
they obtained that
\begin{equation}\label{DMYF}
	\mathbb{E}\exp\Big\{-\frac12\int_0^{\infty}X_u f(u) du\Big\} = \exp\Big\{\frac{x}{2}\Phi'(0) + \frac12\int_0^{\infty}\frac{\Phi'(u)}{\Phi(u)}\Delta(u)du\Big\},
\end{equation}
where $\Phi'' = \Phi f$ and this solution is nonnegative decreasing and $\Phi(0) = 1$. However no other concrete formulae for (\ref{DMYF}) was presented there.
Observe  that taking $\Delta(t) = \delta > 0$ in Corrolary \ref{GBPP} we obtain the Laplace transform of standard squared Bessel process - see   \cite[Corollary 1.3]{RY}.
Faraud and Goutte \cite{FG}
discussed existence, uniqueness  and positivity of solution of generalized squared Bessel process II using the standard techniques (see Section 1.2.1 there). The authors extended the additivity property to the generalized squared Bessel process II. They used the removal drift technique presented earlier in \cite{PY} and in the case of $\theta\equiv 0$, analogous as in \cite{DYM}, they used the solution of Sturm-Liouville equation for a functional of the form $\mathbb{E}\exp\Big\{-\frac12\int_0^{a}X_u f(u)du\Big\}$ for $a>0$ and  presented formula  (\ref{LapTGBI}) for the Laplace transform of the generalized squared Bessel process \cite[Corollary 1.2]{FG}.
 However, they didn't present the Laplace transform for the generalized squared Bessel process given by (\ref{GBPII}).
\end{remark}

As we see in Corollary \ref{jj31} the direct use of Theorem \ref{tw:Lap} give the  Laplace transform of  generalized Bessel processes for very narrow class of coefficients. To obtain a more general result we modify  the matching diffusion method. Let us observe that $p(t,\lambda) = \mathbb{E}e^{-\lambda\hat{X}_t}$ satisfies the following linear PDE of the first order
\begin{align}\label{LaP}
	\frac{\partial p}{\partial t} = \Big(-2\lambda^2 + 2\lambda\theta(t)\Big)\frac{\partial p}{\partial \lambda} -\Delta(t)\lambda p, \quad  p(0,\lambda) = e^{-x\lambda},
\end{align}
where  $\Delta:[0,\infty)\rightarrow [0,\infty)$ and $\theta:[0,\infty)\rightarrow \mathbb{R}$ are given continuous functions and $x\geq 0$.
So, the Laplace transform for process being a member of family of generalized Bessel processes II satisfies a  PDE of a very  special form.  Lemma \ref{unique} guarantees the uniqueness of its solution.
Idea which enables us to find the form of $p$ comes from Remark  \ref{uw2.8}, 
which suggests the form of the solution.

\begin{proposition}\label{GSBPII}
(Generalized squared Bessel process II) Let $\Delta:[0,\infty)\rightarrow [0,\infty)$ and $\theta:[0,\infty)\rightarrow (-\infty,0]$ be continuous functions and $\hat{X}$ be a generalized Bessel process II with $X_0=x\geq 0$. Then, for $\lambda\geq 0$,
\begin{align*}
	\mathbb{E}e^{-\lambda\hat{X}_t} &= e^{-x\phi(t,\lambda)-\psi(t,\lambda)},
\end{align*}
where
\begin{align*}
	\phi(t,\lambda) = \frac{\lambda}{k(t) +2\lambda\int_0^tk(u)du}&,\quad
	\psi(t,\lambda) = v\Big(t, \frac{\lambda}{k(t) +2\lambda\int_0^tk(u)du}\Big),\\
	k(t) = e^{-2\int_0^t\theta(u)du},\quad
	v(t,z) &= z \int_0^t\frac{\Delta(u)k(u)}{1-2z\int_0^uk(s)ds}du.
\end{align*}
\end{proposition}
\begin{proof}
We find a solution $p$ of (\ref{LaP}), giving the Laplace transform, in a form
\begin{equation} \label{j500}
p(t,\lambda) = e^{-x\phi(t,\lambda)-\psi(t,\lambda)}
\end{equation}
with appropriate  chosen $\phi$ and $\psi$.
For a such $p$ we have from (\ref{LaP})
\begin{align*}
	\Big(x\frac{\partial\phi}{\partial t}+\frac{\partial\psi}{\partial t}\Big)p = \Big(-2\lambda^2 + 2\lambda\theta(t)\Big)\Big(x\frac{\partial\phi}{\partial \lambda}+\frac{\partial\psi}{\partial \lambda}\Big)p + \Delta(t)\lambda p.
\end{align*}
As the above equality holds for any $x\geq 0$ we conclude that
\begin{align} \label{eq:fi}
 \frac{\partial\phi}{\partial t} &= \Big(-2\lambda^2 + 2\lambda\theta(t)\Big)\frac{\partial\phi}{\partial \lambda}, \quad \phi(0,\lambda) = \lambda, \\ \label{eq:pfi}
 \frac{\partial\psi}{\partial t} &= \Big(-2\lambda^2 + 2\lambda\theta(t)\Big)\frac{\partial\psi}{\partial \lambda} + \Delta(t)\lambda, \quad \psi(0,\lambda) = 0.
\end{align}
Let us solve \eqref{eq:fi}. The form of this linear PDE can be find in \cite{PZM} (or in 
Eqworld  \cite[Section 1.2 point 9]{E}) and its solution is of the form
$$
	\phi(t,\lambda) = G\Big( \frac{k(t)}{\lambda} + 2\int_0^tk(u)du\Big),
$$
where $G$ is an arbitrary function. Initial condition $\phi(0,\lambda) = \lambda$ and $k(0) =1$ force $G$ to be $G(\lambda) = 1/\lambda$ for $\lambda>0$. \\
Now we  solve \eqref{eq:pfi}. This one can not be find  in \cite{PZM} or  \cite{E}, so we give a solution.
We use the characteristic method and first solve ODE: $q'(t) = 2q^2(t) - 2q(t)\theta(t)$, $q(0) = z$. In result, we obtain
$$
	q_{z}(t) = \frac{z k(t)}{1 - 2z \int_0^tk(u)du}.
$$
Substituting $q_z(t) = \lambda$ we have
$$
	z = \frac{\lambda}{k(t)+2\lambda\int_0^tk(u)du}.
$$
In the second step we solve ODE: $v'(t) = \Delta(t)q_{z}(t), \ v(0) = 0$ and obtain
$$
	v_{z}(t) = z\int_0^t\frac{\Delta(u)k(u)}{1-2z\int_0^uk(s)ds}du.
$$
Let us observe that for $z = \lambda(k(t)+2\lambda\int_0^tk(u)du)^{-1}$ we have $1-2z\int_0^uk(s)ds>0$, so $v$ is well defined. Finally
$$
	\psi(t,\lambda) = v_{z}(t)\big|_{z = \lambda(k(t)+2\lambda\int_0^tk(u)du)^{-1}} .
$$
Summing up, we find a solution of \eqref{LaP} in a form \eqref{j500} and by uniqueness of solution it follows that this solution gives the Laplace transform.
\end{proof}
Now we present an application of Proposition \ref{GSBPII} with $\theta \not\equiv 0$.
\begin{example} (Squared Bessel bridge)\label{SBBB}
A squared Bessel bridge is the unique solution to the SDE
\begin{equation}\label{SBB}
	X_t = x+ 2\int_0^t\sqrt{X_s}dB_s +\int_0^t\Big(\delta - \frac{2X_s}{1-s}\Big)ds,
\end{equation}
where $t\in[0,1)$, $x\geq0$ and $\delta \geq 0$ (see \cite[Chapter XI Section 3]{RY}).

The coefficients for the matching diffusions technique are
$$
\sigma\equiv 0,\  \mu(t,\lambda) = -2\lambda\Big(\lambda + \frac{1}{1-t}\Big).
$$
It turns out that in the case of squared Bessel bridge the matching diffusion technique fails to work, since
the pair $(0, \mu)$ does not belong
to $\mathcal{FK}$-class. Indeed, it is a consequence of Proposition \ref{GBPPII} since the function
$\theta:[0,1)\rightarrow (-\infty,0]$ is given by $\theta(t) = -\frac{1}{1-t}$.
However, using  Proposition \ref{GSBPII} for that $\theta$ and $\Delta(t) \equiv\delta$ we easily obtain the Laplace transform of squared Bessel bridge  presented earlier by Revuz and Yor  \cite[Chapter XI Section 3 ex. (3.6)]{RY}:
\begin{equation}
	\mathbb{E}e^{-\lambda X_t} = \Big(1+2\lambda t (1-t)\Big)^{-\delta/2}e^{-\frac{x\lambda(1-t)^2}{1+2\lambda t(1-t)}}.
\end{equation}
\end{example}

After revisiting the generalized squared Bessel process family, we use the matching diffusion technique to go beyond the affine processes. Let us consider the process given by
\begin{equation}\label{baff}
	\hat{X}_t = x + \int_0^t\sqrt{|\hat{X}_u^2+2\hat{X}_u|}dB_u +\int_0^t(b\hat{X_u} + a)du,
\end{equation}
where $x\geq 0, a\geq 0, b\in\mathbb{R}$. It is an exercise to check that this SDE has an unique nonnegative solution. After skipping the absolute value sign in (\ref{baff}), we have the following result.
\begin{proposition} \label{proposition2.14}
 Let $\hat{X}$ be given by (\ref{baff}) with $b\geq 1/2$. Then, for a fixed $t> 0$,
\begin{align*}
\hat{X}_t  \stackrel{(law)}{=} \frac{R_1}{2}\int_0^te^{B_u + (b-1/2)u}du,
\end{align*}
where $R$ is a $2a$-dimensional squared Bessel process starting from $$2xe^{B_t+(b-1/2)t}\Big/ \int_0^te^{B_u + (b-1/2)u}du$$ and independent of $B$, i.e. for $s\geq0$
$$
	R_s = \frac{2xe^{B_t+(b-1/2)t}}{\int_0^te^{B_u + (b-1/2)u}du} + 2\int_0^s\sqrt{R}_udV_u + 2as,
$$
where $V$ is a standard Brownian motion independent of $B$.
\end{proposition}
\begin{proof}
We use the matching diffusion technique for
$$
	\sigma(\lambda) = \lambda, \ \mu(\lambda) = b\lambda - \lambda^2, \ V(\lambda) = a\lambda.
$$
We have to prove that the pair $(\sigma,\mu)$ is in $\mathcal{FK}$-class for a diffusion
\begin{equation*}
	X_t = \frac{\lambda e^{B_t+(b-1/2)t}}{1+\lambda\int_0^te^{B_u + (b-1/2)u}du} > 0.
\end{equation*}
The process $X$ is the unique solution of
$$
dX_t= (bX_t - X_t^2)dt + X_t dB_t .
$$
As $\mu,\sigma$ do not depend on time and $V$ is a linear function, it suffices to prove that  the following condition is satisfied: 
for all $s\leq t\leq T$, $m\geq 1$
\begin{equation}\label{sup0}
	\mathbb{E}\sup_{u\in[s,t]}X_u^{2m} \leq C (1+\lambda^{2m})e^{C(t-s)},
\end{equation}
where $C = C(m,T)$ is a constant and $X_0 = \lambda$. Indeed, the function $p$ is bounded by 1, so our assumptions and (\ref{sup0}) allow to mimic successfully the proof of Feynman-Kac representation given in  \cite[Theorem 7.6 Section 5]{KS}.
 It remains to prove  condition (\ref{sup0}).
We have
\begin{align*}
	\mathbb{E}\sup_{u\in[s,t]}X_u^{2m} &\leq \lambda^{2m}e^{-(2b-1)sm}\mathbb{E} \Big(e^{sup_{u\in[s,t]}B_u}\Big)^{2m}\\
	&\leq 2\lambda^{2m}e^{m^2t-(2b-1)s} \leq C e^{(2b-1)(t-s)}.
\end{align*}
Thus, by Theorem \ref{tw:Lap}, we obtain
\begin{equation}\label{EL}
	\mathbb{E}e^{-\lambda\hat{X}_t} = \mathbb{E}e^{-xX_t - a\int_0^tX_udu}.
\end{equation}
Observe that
$$
	e^{\int_0^tX_udu} = 1 + \lambda\int_0^te^{B_u + (b-1/2)u}du.
$$
Thus
\begin{align*}
	\mathbb{E}e^{-xX_t - a\int_0^tX_udu} &= \mathbb{E} e^{-xX_t} \Big(1 + \lambda\int_0^te^{B_u + (b-1/2)u}du\Big)^{-a}\\
	&= \mathbb{E}e^{-\lambda \frac{R_1}{2}\int_0^te^{B_u + (b-1/2)u}du},
\end{align*}
where we used the scaling property of squared Bessel process (see  \cite[Prop.1.6 Chapter XI]{RY}).
The assertion follows from the last equality and (\ref{EL}).
\end{proof}

\section{Matching Diffusion for a choosen function and Laplace transform of a vector}
In Theorem \ref{tw:Lap} the matching diffusion technique was presented and used for 
the function $p$ being  the Laplace transform of $\hat{X}$. We now present that the matching diffusion technique can be extended successfully for an appropriately chosen function $f$ and function $p$ of the form $p(t,\lambda):= \mathbb{E}e^{-\lambda f(\hat{X}_t)}$.
\begin{definition}(Hypothesis H) For a given diffusion $\hat{X}$ with coefficients $(\hat{\sigma},\hat{\mu})$ we say that $f\in\mathcal{C}^2$ satisfies hypothesis H if for all $t\geq 0, \lambda \geq 0$
\begin{enumerate}
\item[(i)] $\mathbb{E}e^{-\lambda f(\hat{X}_t)} < \infty$,
\item[(ii)] $ \mathbb{E}\int_0^te^{-2\lambda f(\hat{X}_u)}(f'(\hat{X}_u))^2\hat{\sigma}^2(u,\hat{X}_u)du < \infty,$
\item[(iii)] $\mathbb{E}\int_0^te^{-\lambda f(\hat{X}_u)}\Big|\frac12f''(\hat{X}_u)\hat{\sigma}^2(u,\hat{X}_u)+f'(\hat{X}_u)\hat{\mu}(u,\hat{X}_u)\Big|du < \infty,$
\item[(iv)] $ \frac{\partial^i}{\partial \lambda^i}\mathbb{E}e^{-\lambda f(\hat{X}_t)} = (-1)^i\mathbb{E}f^i(\hat{X}_t)e^{-\lambda f(\hat{X}_t)} < \infty, \ i =1,2.$
\end{enumerate}
\end{definition}
Notice that hypothesis H holds, for example, for a diffusion given by \eqref{tjX-hat} satisfying (\ref{martprop}) and for a  nonnegative function $f\in\mathcal{C}^2$ with bounded both derivatives.
\begin{theorem} \label{Lapf} Let $\hat{X}, X$ be diffusions given by \eqref{tjX-hat} and \eqref{tjX}, respectively.
Let $f$ satisfy hypothesis H and be such that $\mathbb{E}e^{-f(x) X_t}<\infty$ for any $t\geq 0, \lambda > 0$.
 Assume that for a continuous and nonnegative function  $V$ and the  function $h(\lambda) = \mathbb{E}e^{-\lambda f(x)}$, the pair $(\sigma,\mu)$ is in $\mathcal{FK}$-class.
  If, for $\lambda \geq 0$,
\begin{align}
\begin{split}
\label{fassump}
\mathbb{E}&\Big[e^{-\lambda f(\hat{X}_t)}\Big[\frac{\lambda^2}{2}(f'(\hat{X}_t))^2\hat{\sigma}^2(t,\hat{X}_t)-\lambda\Big(f'(\hat{X}_t)\hat{\mu}(t,\hat{X}_t)+\frac12 f''(\hat{X}_t)\hat{\sigma}^2(t,\hat{X}_t)\Big)\Big]\Big]\\
&= \mathbb{E}\Big[e^{-\lambda f(\hat{X}_t)}\Big(\frac{\sigma^2(t,\lambda)}{2}f^2(\hat{X}_t) - \mu(t,\lambda)f(\hat{X}_t) - V(t,\lambda)\Big)\Big], 
\end{split}
\end{align}
then
\begin{equation}
	\mathbb{E}e^{-\lambda f(\hat{X}_t)} = \mathbb{E}e^{-f(x) X_t -\int_0^tV(t-u,X_u)du}.
\end{equation}
\end{theorem}
\begin{proof} For $p(t,\lambda):= \mathbb{E}e^{-\lambda f(\hat{X}_t)}$, It\^o's lemma and assumptions on $f$ give
\begin{align*}
	\frac{\partial p}{\partial t}(t,\lambda) = \mathbb{E}e^{-\lambda f(\hat{X}_t)}\Big[\frac{\lambda^2}{2}(f'(\hat{X}_t))^2\hat{\sigma}^2(t,\hat{X}_t)-\lambda\Big(f'(\hat{X}_t)\hat{\mu}(t,\hat{X}_t)+\frac12 f''(\hat{X}_t)\hat{\sigma}^2(t,\hat{X}_t)\Big)\Big].
\end{align*}
The last identity along with assumption (\ref{fassump}) results in the  following Cauchy problem for $p$
\begin{equation*} \label{cau}
	\frac{\partial p}{\partial t}(t,\lambda) = \mathcal{A}_t p(t,\lambda) - V(t,\lambda)p(t,\lambda), \quad p(0,\lambda) = e^{-\lambda f(x)},
\end{equation*}
where $\mathcal{A}_t$ is the generator of $X$.
The assumption that  $(\sigma,\mu)$ is in $\mathcal{FK}$-class
finishes the proof.
\end{proof}
\begin{remark}\label{techremark-pod}
The following condition is sufficient for (\ref{fassump}) to hold:\\
 for every $x\geq 0,t\geq 0$ and $\lambda \geq 0$
 	\begin{align}
 	\begin{split}
 \label{pde}
  \frac{\lambda^2}{2}&f'(x)^2\hat{\sigma}^2(t,x)-\lambda\Big(f'(x)\hat{\mu}(t,x)+\frac12 f''(x)\hat{\sigma}^2(t,x)\Big) \\
&= \frac{\sigma^2(t,\lambda)}{2}f^2(x) - \mu(t,\lambda)f(x) - V(t,\lambda). 
 \end{split}
 	\end{align}
 \end{remark}
Let us recall that for $x>0,c >0$ the process
\begin{equation} \label{PGSCS}
X_t = x + \int_0^tX_udB_u - c\int_0^tX_u^2du
 \end{equation}
 is known in literature as a population growth in stochastic crowded environment (PGSCE), see   \cite[Ex. 5.15]{O}. 
 The properties and
difficulties associated with the study of distribution of PGSCE are described in \cite{JWII}. Using 
Theorem \ref{Lapf} we have
\begin{corollary} If $X$ is an PGSCE then for $\lambda\geq 0$
\begin{equation*}
	\mathbb{E}e^{-\frac{\lambda}{X_t}} = \mathbb{E}e^{-x\lambda e^{B_t+t/2}-cx\lambda\int_0^te^{B_u +u/2}du}.
\end{equation*}
\end{corollary}
\begin{proof} It follows from Theorem \ref{Lapf} with $\sigma(\lambda) = \mu(\lambda) = \lambda$, $V(\lambda) = c \lambda$ and $f(x) = x^{-1}$.
\end{proof}
From the last theorem we obtain an interesting result about marginals of some special diffusions. It turns out that for a diffusion $\hat{X}$ with coefficients satisfying a simple relation, presented below, we can find a function $f$ such that $f(\hat{X})$ is associated to a geometric Brownian motion. In fact it is a special case of Zvonikin's observation and the Doss-Sussmann method - for details see \cite[Sec. 27 Chapter 5]{RW}. We recall that a process $Y$ is associated to  a process $Z$, if $Y_t$ and $Z_t$ have the same distribution for all $t$.

\begin{proposition}\label{misigma}  Assume that coefficients of  diffusion $\hat{X}$ do not depend on time, satisfy
(\ref{martprop}) and
$\hat{\sigma}$ is a function in   $\mathcal{C}^1$ such that $|\hat{\sigma}|>0$. Let
\begin{equation}\label{1/8}
f(x) = e^{\int_a^x(\hat{\sigma}(z))^{-1}dz}
\end{equation}
be well defined for an arbitrarily chosen
point $a$. Assume that  $\hat{X}_0=x$. If for every $y\geq 0$
\begin{equation}\label{Dossalike}
	\hat{\mu}(y) = \frac{\hat{\sigma}(y)}{2}(1+\hat{\sigma}'(y)),
\end{equation}
and $f$ satisfies hypothesis H,
then $(f(\hat{X}_t),t\geq 0)$ is associated to $(f(x)e^{B_t+t/2},t\geq 0)$.
\end{proposition}
\begin{proof} Observe that for $f$ given by \eqref{1/8} and $\hat{\mu},\hat{\sigma}$ satisfying (\ref{Dossalike}) we have
\begin{align*}
	f = f' \hat{\mu} + \frac12 f'' \hat{\sigma}^2, \quad 	f = f' \hat{\sigma}.
\end{align*}
Thus, for $\lambda \geq 0, x \geq 0$,
\begin{align*}\label{diffcon}
	\frac{\lambda^2}{2}(f'(x))^2 \hat{\sigma}^2(x) - \lambda\Big(f'(x)\hat{\mu}(x) +\frac12 f''(x)\hat{\sigma}^2(x) \Big) = \frac{\lambda^2}{2}f^2(x) - \lambda f(x).
\end{align*}
So \eqref{pde}
is satisfied for $\sigma = \mu = id$ and $V\equiv 0$.
The assertion follows from Theorem \ref{Lapf}.
\end{proof}
Proposition \ref{misigma} gives a way to construct diffusions, which are associated to a geometric Brownian motion.
Usually having the form of this diffusion one can easily check that it is indeed a geometric Brownian motion. This fact is illustrated in the next example.
\begin{example}
Let $\hat{X}$ be given by
$$
	\hat{X}_t = x + 2\int_0^t\sqrt{\hat{X}_u}dB_u + \int_0^t \Big(\sqrt{\hat{X}_u}+1\Big)du .
$$
Then the function $f$ defined by \eqref{1/8} has the form $f(x) = e^{\sqrt{x}}$. It follows from Proposition \ref{misigma} that the process $Y_t = e^{\sqrt{\hat{X}_t}}$ is associated to a geometric Brownian motion.
\end{example}
Having established Proposition \ref{misigma}, the natural question arises - whether the equality between processes $(f(\hat{X}_t),t\geq 0)$ and $(f(x)e^{B_t+t/2},t\geq 0)$ holds only on the level of its marginals or it can be extended on the level of processes.
It turns out that
the equality holds on the level of processes. Namely
\begin{proposition}\label{eqprocess} Let $\hat{X}$ be a diffusion as in Proposition \ref{misigma} and $f$ be defined by \eqref{1/8}.
If \eqref{Dossalike} holds, then
\begin{equation}
(f(\hat{X}_t),t\geq 0) = (f(x)e^{B_t+t/2},t\geq 0).
\end{equation}
\end{proposition}
\begin{proof} It is enough to observe that
(\ref{Dossalike}) and It\^o's lemma for $f(\hat{X}_t)$ imply
\begin{align*}
	df(\hat{X}_t) = f(\hat{X}_t)(dB_t + dt).
\end{align*}
\end{proof}
\begin{remark} It is known
that for $\sigma$ in $\mathcal{C}^2$ with bounded first and second derivatives and a some Lipschitz-continuous function $b$, the  SDE
\begin{equation} \label{Xeq}
	dX_t = \sigma(X_t)dB_t + \Big(b(X_t)+ \frac12 \sigma(X_t)\sigma'(X_t)\Big)dt
\end{equation}
has  a unique strong solution, which  can be written in the form
\begin{equation}
	X_t(\omega) = u(B_t(\omega),Y_t(\omega)),
\end{equation}
where $\omega\in\Omega$, $u$ is a some suitable continuous function, and the process $Y$ solves an ordinary differential equation for every $\omega\in\Omega$. In fact,
\begin{align*}
	\frac{\partial u}{\partial x} &= \sigma(u), \ u(0,y) = y,\\
	\frac{d}{dt}Y_t(\omega) &= h(B_t(\omega),Y_t(\omega)), \ Y_0(\omega) = X_0,
\end{align*}
where
$$  h(x,y) = b(u(x,y))e^{-\int_0^x\sigma'(u(z,y))dz}.
$$
For details see  \cite[Prop. 2.21 Chapter 5]{KS} or  \cite[Section 27 Chapter 5]{RW}. Theorem \ref{eqprocess}
gives a stronger result for $b = \sigma/2$. Namely, if $|\sigma| > 0$, then the unique strong solution of (\ref{Xeq}) is given by
\begin{equation*}
	(X_t,t\geq 0) = \Big(f^{-1}\Big(f(X_0)e^{B_t+t/2}\Big), t\geq 0\Big),
\end{equation*}
where $f(x) = e^{\int_a^x(\sigma(z))^{-1}dz}$.
\end{remark}
Next we use the matching diffusions technique to establish the Laplace transform of vector $(f(\hat{X}_t),\int_0^tf(\hat{X}_u)du)$ for some  functions $f$.
\begin{theorem}\label{ThV} Let $\hat{V}$ be a continuous nonnegative function and $\gamma \geq 0$. Under the assumptions of Theorem \ref{Lapf}, if  $\mathbb{E}\int_0^te^{-\lambda f(\hat{X}_u)}\hat{V}(\hat{X}_u)du<\infty$ for all $t \geq 0$
and the following identity
holds
\begin{align} \label{fassumpV}
\mathbb{E}&e^{-\lambda f(\hat{X}_t)}\Big[\frac{\lambda^2}{2}(f'(\hat{X}_t))^2\hat{\sigma}^2(t,\hat{X}_t)-\lambda\Big(f'(\hat{X}_t)\hat{\mu}(t,\hat{X}_t)+\frac12 f''(\hat{X}_t)\hat{\sigma}^2(t,\hat{X}_t)\Big)-\gamma \hat{V}(\hat{X}_t)\Big]\\
&= \mathbb{E}e^{-\lambda f(\hat{X}_t)}\Big[\frac{\sigma^2(t,\lambda)}{2}f^2(\hat{X}_t) - \mu(t,\lambda)f(\hat{X}_t) - V(t,\lambda)\Big]\notag,
\end{align}
then
\begin{equation}\label{DLT}
\mathbb{E}e^{-\lambda f(\hat{X}_t)-\gamma\int_0^t\hat{V}(\hat{X}_u)du} = \mathbb{E}e^{-f(x)X_t -\int_0^tV(t-u,X_u)du}.
\end{equation}
\end{theorem}
 \begin{proof} Define $p(t,\lambda) = \mathbb{E}e^{-\lambda f(\hat{X}_t)-\gamma\int_0^t\hat{V}(\hat{X}_u)du}$. The proof goes in the same way as in Theorem \ref{Lapf}.
 From It\^o's lemma and assumptions we obtain
 \begin{align*}
 	\frac{\partial p}{\partial t} = \mathcal{A}_t p - Vp, \ p(0,\lambda) = e^{-\lambda f(x)},
 \end{align*}
 where $\mathcal{A}_t$ is the generator of $X$. Now, the proof follows
 from assumption that  $(\sigma,\mu)$ is in $\mathcal{FK}$-class.
 \end{proof}
 \begin{remark}\label{techremark}
The following condition is sufficient for (\ref{fassumpV}) to hold:\\
 For every $x\geq 0,t\geq 0$ and $\lambda \geq 0$
 	\begin{align}\label{pdeV}
 	 \frac{\lambda^2}{2}&f'(x)^2\hat{\sigma}^2(t,x)-\lambda\Big(f'(x)\hat{\mu}(t,x)+\frac12 f''(x)\hat{\sigma}^2(t,x)\Big)-\gamma \hat{V}(x)\\
&= \frac{\sigma^2(t,\lambda)}{2}f^2(x) - \mu(t,\lambda)f(x) - V(t,\lambda). \notag
 	\end{align}
 \end{remark}
 \begin{remark}\label{affine}
In Theorem \ref{ThV} we present the Laplace transform of a vector $(f(X_t),\int_0^t V(X_u)du)$ for some functions $f,V$ and diffusions $X$.
In particular, a special case of this considerations is the problem of finding the bond prices in models with an affine term structure of interest rate, see e.g. Duffie and Kan \cite{DK}. Connections between affine processes and Riccati equations had been studied in Duffie, Filipovi\'{c} and Schachermayer \cite{DFS}. In \cite{DK}, the modeling an affine term structure relied on the numerical techniques (fourth-order Runge-Kutta method). In Boyle, Tian and Guan \cite{BTG}, again in context of an affine term structure, the Kovacic's algorithm had been implemented to obtain analytical formula in a Liouvillian extension of the rational function field sense.
  Our methodology, touching Riccati equations as well, allows elegantly revisit the subject of affine processes (\cite{DFS}) and give closed formulae for Laplace transforms differ from that obtained in \cite{DK} and \cite{BTG}.
\end{remark}
 \begin{corollary} \label{jj1} Let $\hat{X}$ and $f$ be as in Proposition \ref{misigma}. If \eqref{Dossalike} holds,
then
\begin{equation*}
	\mathbb{E}e^{-\lambda f(\hat{X}_t)-\gamma\int_0^tf(\hat{X}_u)du} = \mathbb{E}e^{-f(x)X_t},
\end{equation*}
where
\begin{equation}\label{X}
	X_t = e^{B_t +t/2}\Big(\lambda + \gamma\int_0^te^{-B_s-s/2}ds\Big).
\end{equation}
 \end{corollary}
 \begin{proof} The proof follows from Theorem \ref{ThV} with $\sigma(\lambda) = \lambda$, $\mu(\lambda) =\gamma + \lambda$, $V\equiv 0$ and $\hat{V} = f$. Then
 \begin{align*}
 	dX_t = X_td(B_t +t) +\gamma dt, \ X_0 = \lambda,
 \end{align*}
and this simple linear SDE has the solution given by (\ref{X}).
 \end{proof}
 \begin{remark}
 From Corollary \ref{jj1} or  directly from Proposition \ref{eqprocess} we have, for fixed $t\geq 0$,
\begin{equation*} \label{Xlaw}
	\Big( f(\hat{X}_t),\int_0^tf(\hat{X}_u)du\Big) \stackrel{(law)}{=}\Big(f(x)e^{B_t +t/2}, f(x)\int_0^te^{B_s+s/2}ds\Big).
\end{equation*}
  \end{remark}
 \begin{example} Let us consider  the PGSCE process $\hat{X}$ given by \eqref{PGSCS}. We apply Theorem \ref{ThV} for $f(x)= 1/x$, $\hat{V} = f$, $\mu(\lambda) = \lambda +\gamma$, $\sigma(\lambda) = \lambda$ and
 $V(\lambda) = c\lambda$. Observe that for these functions (\ref{pdeV}) holds. Then, we have
 \begin{equation*}\label{DLTPG}
 \mathbb{E}e^{-\lambda/ \hat{X}_t-\gamma\int_0^t(\hat{X}_u)^{-1}du} = \mathbb{E}e^{-X_t/x -c\int_0^tX_udu},
 \end{equation*}
 where $X$ is given by  (\ref{X}). Hence
 \begin{align*}
 	&\Big((\hat{X}_t)^{-1}, \int_0^t(\hat{X}_u)^{-1}du\Big)\stackrel{(law)}{=}  \\
 	\Big( \frac{e^{B_t +\frac{t}{2}}}{x}+ c\int_0^te^{B_s+\frac{s}{2}}ds,\ & \frac{e^{B_t +\frac{t}{2}}}{x}\int_0^te^{-B_s-\frac{s}{2}}ds +c\int_0^te^{B_u+\frac{u}{2}}\int_0^ue^{-B_s-\frac{s}{2}}dsdu\Big).\notag
 \end{align*}
 \end{example}
Let us  observe that if $\hat{V} = f$ and condition (\ref{pde}) is satisfied for  $\hat{\sigma},\hat{\mu},\sigma$ and $\mu$, then  (\ref{pdeV}) is satisfied for $\hat{\sigma},\hat{\mu},\sigma$ and $\mu + \gamma$. This observation leads us to conclusion that to obtain the PDE for Laplace transform of vector $(\hat{X}_t,\int_0^t\hat{X}_udu)$, where $\hat{X}$ belongs to the family of generalized squared Bessel processes
we have to change $\mu$ by adding a constant $\gamma> 0$ in PDE for Laplace transform of $\hat{X}_t$. In result, if $p(t,\lambda) = \mathbb{E}e^{-\lambda \hat{X}_t-\gamma\int_0^t\hat{X}_udu}$, then
\begin{align}\label{vectorLap}
		\frac{\partial p}{\partial t} = \Big(\gamma-2\lambda^2 + 2\lambda\theta(t)\Big)\frac{\partial p}{\partial \lambda} -\Delta(t)\lambda p, \  p(0,\lambda) = e^{-x\lambda}, \ x\geq 0.
\end{align}
If conditions (\ref{techcon1}) and (\ref{techcon1b}) from Proposition \ref{CCPI} are satisfied for $\mu(t,\lambda) = \gamma-2\lambda^2 + 2\lambda\theta(t)$ and $V(t,\lambda) = \Delta(t)\lambda$, then we can use matching diffusion technique. However, it may happen that $(0,\mu)\notin\mathcal{FK}$. Below we omit this difficulty and present a proposition
giving Laplace transform of vector $(\hat{X}_t,\int_0^t\hat{X}_udu)$ for $\hat{X}$ belonging to the generalized squared Bessel process family. Having the conjecture of the form of this Laplace transform, we are able using techniques exploited before to obtain close formulae for vector's Laplace transforms.

\begin{proposition}\label{Lapvec} Let $\Delta:[0,\infty)\rightarrow [0,\infty)$ and $\theta:[0,\infty)\rightarrow (-\infty,0]$ be continuous functions and $\hat{X}$ be a generalized squared Bessel process II given by (\ref{GBPII}). Fix $\gamma>0$. Let
$y_{x_0}$ be a solution of general Riccati equation
\begin{align*}
	y'(t) = 2y^2(t) - 2y(t)\theta(t) - \gamma, \ y(0) = x_0\in\mathbb{R}.
\end{align*}
Let $ (t,\lambda)\mapsto x_0(t,\lambda)$ be a function  in $\mathcal{C}^1$ such that $y_{x_0}(t)\Big|_{x_0 = x_0(t,\lambda)}=\lambda$ and $|\frac{\partial x_0(t,\lambda)}{\partial\lambda}| > 0$. If the  function $x_0\mapsto y_{x_0}(t)$ is in $\mathcal{C}^1$,  then the unique solution of (\ref{vectorLap}), i.e. the joint Laplace transform of vector $(\hat{X}_t,\int_0^t\hat{X}_udu)$ for $\hat{X}$ being a generalized squared Bessel process, is given by
\begin{align*}
	p(t,\lambda) = e^{-x\phi(t,\lambda)-\psi(t,\lambda)},
\end{align*}
where
\begin{align*}
	\phi(t,\lambda) = x_0(t,\lambda), \ \psi(t,\lambda) = \int_0^t\Delta(u)y_{x_0}(u)du\Big|_{x_0 = x_0(t,\lambda)}.
\end{align*}
\end{proposition}
\begin{proof} Uniqueness of solution of (\ref{vectorLap}) follows from Lemma \ref{unique}, so it is enough to prove that $p(t,\lambda) = e^{-x\phi(t,\lambda)-\psi(t,\lambda)}$ is the solution of \eqref{vectorLap}. As in the proof of Proposition \ref{GSBPII}, it is enough to prove that
\begin{align} \label{j310}
 \frac{\partial\phi}{\partial t} &= \Big(\gamma-2\lambda^2 + 2\lambda\theta(t)\Big)\frac{\partial\phi}{\partial \lambda}, \quad \phi(0,\lambda) = \lambda, \\
 \label{j320}
 \frac{\partial\psi}{\partial t} &= \Big(\gamma-2\lambda^2 + 2\lambda\theta(t)\Big)\frac{\partial\psi}{\partial \lambda} + \Delta(t)\lambda, \quad \psi(0,\lambda) = 0.
\end{align}
Since $\phi(t,\lambda) = x_0(t,\lambda)$, by  definition of $\phi$, we have $\frac{\partial\phi(t,\lambda)}{\partial t} = \frac{\partial x_0(t,\lambda)}{\partial t}$ and $\frac{\partial\phi(t,\lambda)}{\partial \lambda} = \frac{\partial x_0(t,\lambda)}{\partial \lambda}$.
 Let $u(t,z) = y_{z}(t)$. Then
$u(t,x_0(t,\lambda)) = \lambda$ and
\begin{align*}
&\frac{\partial u (t,x_0(t,\lambda))}{\partial t} + \frac{\partial u(t,x_0(t,\lambda))}{\partial x_0}\frac{\partial x_0(t,\lambda)}{\partial t} = 0 , \\
&\frac{\partial u(t,x_0(t,\lambda))}{\partial x_0}\frac{\partial x_0(t,\lambda)}{\partial \lambda} = 1.
\end{align*}
From the definition of $y_{x_0}$ we have
$$
	\frac{\partial u (t,x_0(t,\lambda))}{\partial t} = y_{x_0}'(t)\Big|_{x_0 = x_0(t,\lambda)} = 2\lambda^2 -2\lambda\theta(t) - \gamma.
$$
Thus, from the last three expressions we obtain
\begin{equation}\label{eqx0}
	\frac{\partial x_0(t,\lambda)}{\partial t} = (\gamma -2\lambda^2 +2\lambda\theta(t))\frac{\partial x_0(t,\lambda)}{\partial \lambda},
\end{equation}
so \eqref{j310} is satisfied. Observe that $u(0,x_0) = y_{x_0}(0) = x_0$ and $x_0(0,\lambda) = u(0,x_0(0,\lambda)) = \lambda$.
Let $v(t,x_0) = \int_0^t\Delta(u)y_{x_0}(u)du$. Then $\psi(t,\lambda) = v(t,x_0(t,\lambda))$ \\
and
\begin{align*}
\frac{\partial \psi (t,\lambda)}{\partial t} = \frac{\partial v (t,x_0(t,\lambda))}{\partial t} + \frac{\partial v(t,x_0(t,\lambda))}{\partial x_0}\frac{\partial x_0(t,\lambda)}{\partial t}.
\end{align*}
Moreover, $\frac{\partial v (t,x_0))}{\partial t} = \Delta(t)y_{x_0}(t)$. In result, from the last two expressions and from (\ref{eqx0}) we obtain
$$
	\frac{\partial \psi (t,\lambda)}{\partial t} = \Delta(t)\lambda + (\gamma -2\lambda^2 +2\lambda\theta(t))\frac{\partial \psi (t,\lambda)}{\partial \lambda}.
$$
It is clear from definition that $\psi(0,\lambda) = 0$. Thus  \eqref{j320} is satisfied, which finishes the proof.
\end{proof}
In the next propositions we consider generalized squared Bessel processes for which assumptions of Proposition \ref{Lapvec} are satisfied.

\begin{proposition}(Generalized squared Bessel process I) \label{GDLP} Let $\lambda\geq 0, \gamma > 0$ and $\hat{X}$ be a generalized squared Bessel process  given by  (\ref{GBP}).
Then
\begin{equation*}
\mathbb{E}e^{-\lambda \hat{X}_t-\gamma\int_0^t\hat{X}_udu} = e^{-x\phi(t,\lambda)-\psi(t,\lambda)},
\end{equation*}
where
\begin{align*}
	\phi(t,\lambda) &= \alpha\Big[1+\frac{2(\lambda-\alpha)}{2\alpha e^{4\alpha t}-(\lambda-\alpha)(1-e^{4\alpha t})}\Big], \quad \alpha = \sqrt{\gamma/2},\\
	\psi(t,\lambda) &= \int_0^t\Delta(u)y_{x_0}(u)du\Big|_{x_0 = \phi(t,\lambda)},\\
	y_{x_0}(t) &= \alpha\Big[1 + \frac{2(x_0-\alpha)e^{4\alpha t}}{2\alpha +(x_0-\alpha)(1-e^{4\alpha t})}\Big].
\end{align*}
\end{proposition}
\begin{proof}
The result follows from Proposition \ref{Lapvec}. We are using notation form that proposition. For $\gamma = 2\alpha^2$ we solve PDE
$$
	y' = 2y^2-2\alpha^2, \ y(0) = x_0.
$$
It is a special Riccati equation (see \cite{PZ2} or \cite[Section 1 point 19]{E}) which is solved by
$$
	y_{x_0}(t) = \alpha\Big[1 + \frac{2(x_0-\alpha)e^{4\alpha t}}{2\alpha +(x_0-\alpha)(1-e^{4\alpha t})}\Big].
$$
From that we obtain
$$
	x_0(t,\lambda) = \alpha\Big[1+\frac{2(\lambda-\alpha}{2\alpha e^{4\alpha t}-(\lambda-\alpha)(1-e^{4\alpha t})}\Big].
$$
Observe that
$$
	\frac{\partial x_0(t,\lambda)}{\partial\lambda} = \frac{4\alpha^2 e^{4\alpha t}}{(2\alpha e^{4\alpha t}-(\lambda-\alpha)(1-e^{4\alpha t}))^2} > 0.
$$
The assertion follows.
\end{proof}
 In Part II Section 8 Point 1.9.7 \cite{BS} it is given a formula that enables to obtain the vector Laplace transform of $(\hat{X}_t, \int_0^t\hat{X}_udu)$ for $\hat{X}$ being a squared radial Ornstein-Uhlenbeck process. However, the formula presented there is complicated and includes modified Bessel function. Below, we present  a simpler close formula for this  Laplace transform.
\begin{proposition}
Let be $\hat{X}$ a squared radial Ornstein-Uhlenbeck process given by  (\ref{ROU}). Then, for $\lambda\geq 0, \gamma> 0$,
\begin{align*}
\mathbb{E}e^{-\lambda \hat{X}_t-\gamma\int_0^t\hat{X}_udu} = e^{-x\phi(t,\lambda)-\psi(t,\lambda)},
\end{align*}
where
\begin{align*}
	\phi(t,\lambda) &= a + \frac{(\lambda-a)(2a +\alpha)}{e^{2t(\alpha+2a)}(a+\alpha+\lambda)+a-\lambda}, \quad a = \frac{\sqrt{\alpha^2+2\gamma}-\alpha}{2},\\
	\psi(t,\lambda) &= \delta a t - \frac{\delta}{2}\ln \Big|1+ \frac{(\lambda-a)(1-e^{2t(\alpha+2a)})}{e^{2t(\alpha+2a)}(a+\lambda+\alpha)+a-\lambda}\Big|.
\end{align*}
\end{proposition}
\begin{proof}
 The result again follows from Proposition \ref{Lapvec}. We solve PDE
$$
	y' = 2y^2+2\alpha y-\gamma, \ y(0) = x_0,
$$
which is a Riccati equation  (see \cite{PZ2} or \cite[Section 1 point 15]{E}) with $f = 2, g =2\alpha, a = \frac{\sqrt{\alpha^2+2\gamma}-\alpha}{2}$.
We obtain the solution
$$
	y_{x_0}(t) = a + \frac{e^{2t(\alpha+2a)}(x_0-a)(2a+\alpha)}{2a+\alpha + (x_0-a)(1-e^{2t(\alpha+2a)})}.
$$
Thus
$$
	x_0(t,\lambda) = a + \frac{(\lambda-a)(2a +\alpha)}{e^{2t(\alpha+2a)}(a+\alpha+\lambda)+a-\lambda}
$$
and
$$
	\frac{\partial x_0(t,\lambda)}{\partial\lambda}  = \frac{e^{2t(\alpha+2a)}(2a+\alpha)^2}{(e^{2t(\alpha+2a)}(a+\alpha+\lambda)+a-\lambda)^2}>0.
$$
Finally
\begin{align*}
	\psi(t,\lambda) &= \delta\int_0^ty_{x_0}(u)du\Big|_{x_0 = x_0(t,\lambda)}\\
	&= \delta a t - \frac{\delta}{2}\ln \Big|1+ \frac{(\lambda-a)(1-e^{2t(\alpha+2a)})}{e^{2t(\alpha+2a)}(a+\lambda+\alpha)+a-\lambda}\Big|.
\end{align*}
\end{proof}
Now, we present Laplace transform of the vector $(\hat{X}_t, \int_0^t\hat{X}_udu)$, where $\hat{X}$ is a squared Bessel bridge given by (\ref{SBB}).
\begin{proposition}
Let $\hat{X}$ be a squared Bessel bridge. Then for any $\lambda\geq0,\gamma>0$ and $t\in[0,1]$
\begin{equation*}
\mathbb{E} e^{-\lambda \hat{X}_t-\gamma\int_0^t\hat{X}_udu} = e^{-x\phi(t,\lambda) -\psi(t,\lambda)},
\end{equation*}
where
\begin{align*}
\phi(t,\lambda) &= \sqrt{\frac{\gamma}{2}} - \frac12 + \Big(\frac{e^{2\sqrt{2\gamma}t}}{\lambda +\frac{1}{2(1-t)}-\sqrt{\gamma/2}}-\frac{1}{\sqrt{2\gamma}}(1-e^{2\sqrt{2\gamma}t})\Big)^{-1},\\
\psi(t,\lambda) &= \delta\int_0^ty_{x_0}(u)du\Big|_{x_0 = \phi(t,\lambda)},\\
y_{x_0}(t) &= \sqrt{\frac{\gamma}{2}} - \frac12 + e^{2\sqrt{2\gamma}t}\Big(\frac{1}{x_0 +1/2 - \sqrt{\gamma/2}}+\frac{1}{\sqrt{2\gamma}}(1-e^{2\sqrt{2\gamma}t})\Big)^{-1}.
\end{align*}
\end{proposition}
\begin{proof} We use Proposition \ref{Lapvec}. Let $t\in[0,1)$. We solve PDE
\begin{equation} \label{j330}
	y' = 2y^2+\frac{2y}{1-t}-\gamma, \ y(0) = x_0.
\end{equation}
Let $u(t) = y(1-t)$ and $s = 1-t$. \eqref{j330} is equivalent to
$$
	u' = -2u^2+\frac{2u}{s}+\gamma.
$$
Using substitution  $v(s) = u(s) + \frac{1}{2s}$ we transform the above PDE to
$$
	v' = -2v^2 + \gamma,
$$
which is a Riccati equation (see \cite{PZ2} or \cite[Section 1 point 15]{E}) with $f = -2, g =0, a = \sqrt{\gamma/2}$.
The solution has the form
$$
	v(s) = \sqrt{\gamma/2} + e^{-2\sqrt{2\gamma}s}\Big(C - e^{-2\sqrt{2\gamma}s}/\sqrt{2\gamma}\Big)^{-1},
$$
where $C$ is a constant. Hence
$$
	y_{x_0}(t) = \sqrt{\frac{\gamma}{2}} - \frac12 + e^{2\sqrt{2\gamma}t}\Big(\frac{1}{x_0 +1/2 - \sqrt{\gamma/2}}+\frac{1}{\sqrt{2\gamma}}(1-e^{2\sqrt{2\gamma}t})\Big)^{-1}.
$$
This gives
$$
		x_0(t,\lambda) = \sqrt{\frac{\gamma}{2}} - \frac12 + \Big(\frac{e^{2\sqrt{2\gamma}t}}{\lambda +\frac{1}{2(1-t)}-\sqrt{\gamma/2}}-\frac{1}{\sqrt{2\gamma}}(1-e^{2\sqrt{2\gamma}t})\Big)^{-1}
$$
and
$$
	\frac{\partial x_0(t,\lambda)}{\partial\lambda}  = \frac{-e^{2\sqrt{2\gamma}t}}{\Big(\lambda +\frac{1}{2(1-t)}-\sqrt{\gamma/2}\Big)^2\Big(\frac{e^{2\sqrt{2\gamma}t}}{\lambda +\frac{1}{2(1-t)}-\sqrt{\gamma/2}}-\frac{1}{\sqrt{2\gamma}}(1-e^{2\sqrt{2\gamma}t})\Big)^2} < 0.
$$
Finally
$$
	\psi(t,\lambda) = \delta\int_0^ty_{x_0}(u)du\Big|_{x_0 = \phi(t,\lambda)}.
$$
\end{proof}
 Let $\gamma$ tend to $0$ in the last three propositions. Then we obtain in limit the Laplace transforms presented in Section 2 (see Proposition \ref{SROU}, Corollary \ref{GBPP} and Example \ref{SBBB}).

\section{Matching diffusion for a general function}
In this section we continue the idea of changing role of coordinates
for some well chosen  function $f$. We consider two diffusions $\hat{X}, X$ with state spaces $\hat{E}\subseteq\mathbb{R},E\subseteq\mathbb{R}$  respectively (usually as a state space we take the real line, the half-line or interval). By $A, \hat{A}$ we denote two differential operators coinciding on $\mathcal{C}_K^2$
 with the generators of diffusions $X$ and $\hat{X}$ respectively.
We start from theorem which will be crucial in proving  later results, among others a Feynman-Kac representation for hyperbolic (Theorem \ref{FKhyp}).

\begin{theorem}\label{metagenthe} Let $\hat{X}$ be a diffusion with generator $\hat{A}$, $\hat{X_0}=x$. Assume that $f: E\times\hat{E} \rightarrow \mathbb{R}$ is a $\mathcal{C}^2$ function such that for any $t\geq0$
\begin{equation} \label{con111}
	\mathbb{E}f(\lambda, \hat{X}_t) = f(\lambda,x) + \int_0^t\mathbb{E}\hat{A}f(\lambda,\hat{X}_u )du.
\end{equation}
Let $X$ be a diffusion given by \eqref{tjX}.
Let $V:[0,\infty)\times E\rightarrow[0,\infty)$ be a continuous function and $h(\lambda) = f(\lambda,x)$. Assume that, for these functions, the pair $(\sigma,\mu)$ is in $\mathcal{FK}$-class. Let $p = p(t,\lambda) := \mathbb{E}f(\lambda, \hat{X}_t)$. If
\begin{equation} \label{con222}
	\mathbb{E}\hat{A}f(\lambda,\hat{X}_u) = A p(u,\lambda) - V p(u,\lambda),
\end{equation}
then
\begin{equation}
	\mathbb{E}f(\lambda, \hat{X}_t) = \mathbb{E}f(X_t,x)e^{-\int_0^tV(t-u,X_u)du}.
\end{equation}
\end{theorem}
\begin{proof}
 We have $p(0,\lambda) = f(\lambda, x)$, as $\hat{X}_0 = x$. Assumptions (\ref{con222}) and (\ref{con111})
  imply that $p$ satisfies
\begin{align}
	\frac{\partial p}{\partial t} = A p - Vp, \quad p(0,\lambda) = f(\lambda, x).
\end{align}
Assumption that  $(\sigma,\mu)$ is in $\mathcal{FK}$-class
finishes the proof.
\end{proof}
Now we introduce a class analogue to conditions $\mathcal{FK}$.
\begin{definition}(Hypothesis HF) For a given diffusion $\hat{X}$ with coefficients $(\hat{\sigma},\hat{\mu})$ we say that $f\in\mathcal{C}^2$ satisfies the hypothesis HF if for all $t\geq 0, \lambda \geq 0$
\begin{enumerate}
\item[(i)] $ \mathbb{E}f(\lambda, \hat{X}_t) <\infty$,
\item[(ii)] $ \mathbb{E}\int_0^t\Big(\frac{\partial f}{\partial x}(\lambda,\hat{X}_u)\hat{\sigma}(u,\hat{X}_u)\Big)^2du < \infty,$
\item[(iii)] $\mathbb{E}\int_0^t\Big| \frac{\hat{\sigma}^2(u,\hat{X}_u)}{2}\frac{\partial^2 f}{\partial x^2}(\lambda,\hat{X}_u )+ \hat{\mu}(u,\hat{X}_u)\frac{\partial f}{\partial x}(\lambda,\hat{X}_u ) \Big| du<\infty,$
\item[(iv)] $ \frac{\partial^i}{\partial \lambda^i}\mathbb{E}f(\lambda,\hat{X}_t ) = \mathbb{E}\frac{\partial^i f}{\partial \lambda^i}(\lambda,\hat{X}_t ) < \infty, \ i =1,2.$
\end{enumerate}
\end{definition}

\begin{theorem}\label{genthe} Let $\hat{X}, X$ be diffusions  given  by \eqref{tjX-hat} and \eqref{tjX}, respectively.
 Assume that $f: E\times\hat{E} \rightarrow \mathbb{R}$ satisfies the hypothesis HF for $\hat{X}$. 
Fix $x\in\hat{E}$. Let $V:[0,\infty)\times E\rightarrow[0,\infty)$ be a continuous function and $h(\lambda) = f(\lambda,x)$.  Assume that, for these functions, the pair $(\sigma,\mu)$ is in $\mathcal{FK}$-class.
 If, for any $t\geq 0$, $\lambda\in E$
\begin{align} \label{gencon}
\mathbb{E}&\Big[ \frac{\hat{\sigma}^2(t,\hat{X}_t)}{2}\frac{\partial^2 f}{\partial x^2}(\lambda,\hat{X}_t )+ \hat{\mu}(t,\hat{X}_t)\frac{\partial f}{\partial x}(\lambda,\hat{X}_t ) \Big]\\
	&= \mathbb{E}\Big[\frac{\sigma^2(t,\lambda)}{2}\frac{\partial^2 f}{\partial \lambda^2}(\lambda,\hat{X}_t )+ \mu(t,\lambda)\frac{\partial f}{\partial \lambda}(\lambda,\hat{X}_t ) - V(t,\lambda)f(\lambda,\hat{X}_t ) \Big],\notag
\end{align}
then
\begin{equation}
	\mathbb{E}f(\lambda, \hat{X}_t) = \mathbb{E}f(X_t,x)e^{-\int_0^tV(t-u,X_u)du}.
\end{equation}
\end{theorem}
\begin{proof} Let $\hat{A}$ be a differential operator coinciding on $\mathcal{C}_K^2$ with generator of $\hat{X}$ and $A$ be such an operator for $X$. Then, from the hypothesis HF,
\begin{align}\label{con11}
	\mathbb{E}f(\lambda, \hat{X}_t) = f(\lambda,x) + \int_0^t\mathbb{E}\hat{A}f(\lambda,\hat{X}_u )du.	
\end{align}
Let us recall that $x$ is fixed, as $\hat{X}_0 = x$. Define $p = p(t,\lambda) := \mathbb{E}f(\lambda, \hat{X}_t)$. We have $p(0,\lambda) = f(\lambda, x)$.
From the hypothesis HF and assumption
(\ref{gencon}) we have
\begin{align} \label{con22}
	\mathbb{E}\hat{A}f(\lambda,\hat{X}_u) = A p(u,\lambda) - V p(u,\lambda).
\end{align}
In result from (\ref{con11}) and (\ref{con22})
\begin{align}
	\frac{\partial p}{\partial t} = A p - Vp, \quad p(0,\lambda) = f(\lambda, x).
\end{align}
Assumption that  $(\sigma,\mu)$ is in $\mathcal{FK}$-class
finishes the proof.
\end{proof}
We can replace condition (\ref{gencon}) by a sufficient one.
\begin{remark}\label{loccor} If a function  $f$
satisfies the following PDE
\begin{align*}
	\frac{\hat{\sigma}^2(t,x)}{2}\frac{\partial^2 f}{\partial x^2} + \hat{\mu}(t,x) \frac{\partial f}{\partial x} = \frac{\sigma^2(t,\lambda)}{2}\frac{\partial^2 f}{\partial \lambda^2} + \mu(t,\lambda) \frac{\partial f}{\partial \lambda} - V(t,\lambda)f,
\end{align*}
then \eqref{gencon} holds.
\end{remark}

\begin{proposition} \label{prop4.3}
Let $f$ be a solution of the following Cauchy problem
\begin{align} \label{cauchy}
	\frac12 \frac{\partial^2 f}{\partial \lambda^2} = \frac{\hat{\sigma}^2(t,x)}{2}\frac{\partial^2 f}{\partial x^2} + \hat{\mu}(t,x) \frac{\partial f}{\partial x}, \quad f(0,x) = h(x),
\end{align}
for some $\mathcal{C}^2$ function $h$.
Let $\hat{X}$ be a diffusion given by \eqref{tjX-hat}. Assume that the coefficients $(\hat{\sigma},\hat{\mu})$ are in $\mathcal{FK}$-class and $f$ satisfies the hypothesis HF.
Then, for any $t>0$,
\begin{align*}
	\mathbb{E}h(\hat{X}_t) =  \int_{-\infty}^{\infty}f(z,x)\frac{1}{\sqrt{2\pi t}}e^{-\frac{z^2}{2t}}dz.
\end{align*}
\end{proposition}
\begin{proof}
Let $E = \mathbb{R}$, $X_t = \lambda + B_t$ and $V\equiv 0$. \eqref{gencon}  follows from Remark \ref{loccor} and \eqref{cauchy}, so we can use Theorem \ref{genthe} and obtain
\begin{align*}
	\mathbb{E}f(\lambda, \hat{X}_t) = \mathbb{E}f(\lambda + B_t,x) = \int_{-\infty}^{\infty}f(\lambda + z,x)\frac{1}{\sqrt{2\pi t}}e^{-\frac{z^2}{2t}}dz.
\end{align*}
Letting $\lambda$ go to $0$ we obtain, for every $t> 0$,
\begin{align*}
	\mathbb{E}h(\hat{X}_t) =  \int_{-\infty}^{\infty}f(z,x)\frac{1}{\sqrt{2\pi t}}e^{-\frac{z^2}{2t}}dz.
\end{align*}
\end{proof}
\begin{example} We use the last proposition to compute the expectation $\mathbb{E}H(\hat{X}_t)$ for $\hat{X}$ defined by (\ref{baff}) with $a =b=0, x>0$ and a function $H = H(x), x> 0$ being a solution of
$$
	H''(x) -\frac{\alpha^2}{x^2+2x}H(x) = 0
$$
for $\alpha>0$. Define $f(\lambda,x) = e^{\alpha\lambda}H(x)$. It is a separable particular solution of the generalized Tricomi equation
$$
	\frac{\partial^2 f}{\partial \lambda^2} = (x^2+2x)\frac{\partial^2 f}{\partial x^2}, \quad f(0,x) = H(x)
$$
(see \cite[Section 4.2 point 3]{E}). If $f$ satisfies the hypothesis HF, then from Proposition \ref{prop4.3} we obtain
$$
	\mathbb{E}H(\hat{X}_t) =  H(x)\int_{-\infty}^{\infty}e^{\alpha z}\frac{1}{\sqrt{2\pi t}}e^{-\frac{z^2}{2t}}dz = H(x)e^{\frac{\alpha^2 t}{2}}.
$$
Notice that the set of functions solving the above Tricomi equation and satisfying the hypothesis HF is nonempty - for $\alpha = \sqrt{2}$ one can easily check that $f(\lambda,x) = e^{\sqrt{2}\lambda}(x^2+2x)$ is such a solution.
\end{example}
Now we present a deep and interesting consequence of Theorem \ref{metagenthe}. We deduce a version of probabilistic representation for solution of a Cauchy problem for hyperbolic PDE. The classical Feynman-Kac representation works for the parabolic equations. The matching diffusion technique gives a new perspective.
\begin{theorem}(Feynman-Kac representation for hyperbolic PDE) \label{FKhyp}
Let $x\in\hat{E}= [0,\infty)$ and $\lambda\in E\subseteq\mathbb{R}$. Let $V:[0,\infty)\times E \rightarrow [0,\infty)$ be a continuous function and $h:E\rightarrow\mathbb{R}$ be in $\mathcal{C}^2$.
Suppose that $f  \in \mathcal{C}^2$
is a bounded  solution of the Cauchy problem
\begin{align*}
	\frac12 \frac{\partial^2 f}{\partial x^2} &= \frac{\sigma^2(t,\lambda)}{2}\frac{\partial^2 f}{\partial \lambda^2} + \mu(t,\lambda) \frac{\partial f}{\partial \lambda} - V(t,\lambda)f, \\ 								 f(\lambda,0) &= h(\lambda), \quad \frac{\partial f}{\partial x}(\lambda, 0+) = 0.
\end{align*}
Let $X$ be a diffusion  given by \eqref{tjX} and  $\hat{X}_t = |B_t|$.
If $f$ satisfies the hypothesis HF for $\hat{X}$
 and for a given $V$, $h$ the pair  $(\sigma,\mu)$ is in $\mathcal{FK}$-class, then
\begin{equation}\label{hh}
	f(\lambda,v) = 2\sqrt{(v\gamma)/\pi}\mathcal{L}_{\gamma,\sqrt{v}}^{-1}\Big(\mathbb{E}(h(X_{\frac{1}{2\gamma}})
e^{-\int_0^{\frac{1}{2\gamma}}V(\frac{1}{2\gamma}-u,X_u)du})\Big),
\end{equation}
where  $\mathcal{L}_{\gamma,v}^{-1}$ stands for the inverse Laplace transform with respect to coordinate $v$, i.e.
$\mathcal{L}_{\gamma,v}^{-1}(g(\gamma))$ is the function  $G(\gamma,v)$ such that, for $\gamma>0$,
$$
	\int_0^{\infty}e^{-\gamma v}G(\gamma,v)dv = g(\gamma).
$$
\end{theorem}
\begin{proof}
Assumptions $f\in\mathcal{C}^2_b$ and $\frac{\partial f}{\partial x}(\lambda, 0+) = 0$ guarantee that $f$ with respect to $x$-coordinate belongs to the domain of reflected Brownian motion (see \cite[Appendix 1 p. 2]{BS}). We use Remark \ref{loccor} and Theorem \ref{metagenthe} for $f,\hat{X}$ and $X$.
We obtain
\begin{align*}
	\mathbb{E}f(\lambda, |B_t|) = \mathbb{E}f(X_t,0)e^{-\int_0^tV(t-u,X_u)du} = \mathbb{E}h(X_t)e^{-\int_0^tV(t-u,X_u)du}.
\end{align*}
In result, for a fixed $t>0$,
\begin{align*}
	2\int_{0}^{\infty}f(\lambda,z)\frac{1}{\sqrt{2\pi t}}e^{-\frac{z^2}{2t}}dz = \mathbb{E}h(X_t)e^{-\int_0^tV(t-u,X_u)du}.
\end{align*}
The assertion follows after substitution $\gamma = \frac{1}{2t}$, $v = z^2$ and some simple algebra.
\end{proof}
\begin{remark} Let us recall,
following \cite{KS} or \cite{F}, a set of standard assumptions for the pair $(\sigma,\mu)$, from Theorem \ref{FKhyp}, to be in $\mathcal{FK}$-class.
Given $f\in\mathcal{C}^{1,2}$ and $h\in\mathcal{C}^2$, both the functions of polynomial growth, it is enough that $\sigma,\mu$ are continuous, satisfy the linear growth conditions, and  $V$ is nonnegative and continuous. Conversely, the sufficient conditions under which the Cauchy problem from the definition of $\mathcal{FK}$-class has a solution $f$ satisfying the exponential growth condition are the following: $\sigma,\mu$ are bounded and uniformly H$\ddot{o}$lder-continuous, $V$ is nonnegative and  continuous or bounded and uniformly H$\ddot{o}$lder-continuous. The function $h$ is of polynomial growth. Assumptions on boundedness and  H$\ddot{o}$lder continuity of  $\sigma,\mu$ and $V$ can be relaxed by making them local requirements (see p.366-368 in \cite{KS} and p.139-150 in \cite{F}).
Assumption $\frac{\partial f}{\partial x}(\lambda, 0+) = 0$ is not necessary. It allows to choose as the process  $\hat{X}$ a reflected Brownian motion instead of a standard one and obtain a representation (\ref{hh}) by an inverse Laplace transform.
For a diffusion given by \eqref{tjX-hat} and satisfying (\ref{martprop}), the sufficient conditions ensuring that point (iv) of the hypothesis HF holds is that $f$ is bounded and in $\mathcal{C}^2$.
Points (ii) and (iii) should be  verified.
The solution of Cauchy problem can be relaxed to $f$ being of polynomial growth but then the conditions of the hypothesis HF should  be verified separately.
\end{remark}

\begin{example} \label{ex4.7}
We illustrate an efficiency of Theorem \ref{FKhyp} by finding a bounded and  satisfying the hypothesis HF solution of the following hyperbolic PDE:
\begin{align*}
	\frac12 \frac{\partial^2 f}{\partial x^2} &= \frac{\lambda^2}{2}\frac{\partial^2 f}{\partial \lambda^2} + \frac{\lambda}{2} \frac{\partial f}{\partial \lambda} - \frac{\lambda^2}{2}f, \\ 								 f(\lambda,0) &= e^{-\lambda}, \quad \frac{\partial f}{\partial x}(\lambda, 0) = 0, \quad \lambda>0, \quad x\geq 0.
\end{align*}
We take $\hat{E}=[0,\infty)$, $E = (0,\infty)$ and $X_t = \lambda e^{B_t}$. The assumptions of Theorem \ref{FKhyp} are satisfied
and we can write, for any $\lambda>0$,
\begin{align*}
	\mathbb{E}f(\lambda, |B_t|) = \mathbb{E}h(X_t)e^{-\int_0^tV(t-u,X_u)du}
	= \mathbb{E}e^{-\lambda e^{B_t} - \frac{\lambda^2}{2}\int_0^te^{2B_u}du}
 = \mathbb{E}e^{-\lambda\cosh(B_t)},
\end{align*}
where in the last equality we use Corollary \ref{HBP}. As the function $\cosh$ is an even function, we conclude that $f(\lambda,x) = e^{-\lambda\cosh(x)}$. Observe that $$\frac{\partial f}{\partial x}(\lambda, 0) = -\lambda\sinh(x)e^{-\lambda\cosh(x)}\Big|_{x=0} = 0.$$
So, $f$ is a desired solution.  One can easily verify that that $f$ satisfies  the hypothesis HF for $|B|$.
\end{example}

\begin{proposition}
Let $E= [0,\infty)$ and $\hat{X}$ be a diffusion given by \eqref{tjX-hat}. Suppose that for the diffusion $\hat{X}$ a function $f$ satisfies the hypothesis HF, $V\equiv 0$ and $f$ solves the PDE
\begin{equation} \label{j400}
	\frac{\partial f}{\partial\lambda} = \frac{\hat{\sigma}^2(t,x)}{2}\frac{\partial^2 f}{\partial x^2} + \hat{\mu}(t,x) \frac{\partial f}{\partial x}.
\end{equation}
Assume that  $(\hat{\sigma},\hat{\mu})$ is  in $\mathcal{FK}$-class. Then
\begin{equation*}
	\mathbb{E}f(\lambda,\hat{X}_t) = \mathbb{E}h(\hat{X}_{t+\lambda}),
\end{equation*}
where $h(x) = f(0,x)$.
\end{proposition}
\begin{proof} Fix $t\geq 0$. From assumption that $(\hat{\sigma},\hat{\mu})$ is  in $\mathcal{FK}$-class
with $\lambda$ taking the role of time we obtain
\begin{equation} \label{classic}
	f(\lambda,x) = \mathbb{E}h(\hat{X}_{\lambda}).
\end{equation}
Matching the diffusion, so taking $\mu\equiv 1$ and $\sigma\equiv 0$ and using \eqref{j400} we see that \eqref{gencon} holds.
 So, by Theorem  \ref{genthe} we have
\begin{align*}
	\mathbb{E}f(\lambda, \hat{X}_t) = f(\lambda +t,x),
\end{align*}
and
(\ref{classic}) finishes the proof.
\end{proof}

\begin{example}\label{hl}
Let $\hat{X}, X$ be diffusions given by \eqref{tjX-hat} and \eqref{tjX}. Assume,  additionally, that  $|\hat{\sigma}| >0, |\sigma|>0$ and $\hat{\mu} = -\hat{\sigma}^2/2$, $\mu = \sigma^2/2$.
 Let functions $h$ on $\hat{E}$ and $l$ on $E$ be defined by formulae:
\begin{align*}
	h(x) &= e^x\int_a^xe^{-w}\int_b^w\hat{\sigma}^{-2}(t,z)dzdw, \\
	l(\lambda) &= e^{\lambda}\int_c^{\lambda}e^{-s}\int_d^s\sigma^{-2}(t,u)duds
\end{align*}
for some arbitrarily chosen points $a,b,c,d$ and fixed $t\geq 0$. If $f(x,\lambda)= h(x) + l(\lambda)$ satisfies the assumptions
of the hypothesis HF for $\hat{X}$, then
\begin{equation*}
	\mathbb{E}h(\hat{X}_t) = h(x) - l(\lambda) + \mathbb{E}l(X_t).
\end{equation*}
Indeed, observe that $f$ solves
\begin{align*}
	\frac{\hat{\sigma}^2(t,x)}{2}\frac{\partial^2 f}{\partial x^2} + \hat{\mu}(t,x) \frac{\partial f}{\partial x} = \frac{\sigma^2(t,\lambda)}{2}\frac{\partial^2 f}{\partial \lambda^2} + \mu(t,\lambda) \frac{\partial f}{\partial \lambda},
\end{align*}
as  $h$ is a solution of
\begin{align*}
	y'' = y' + 2\hat{\sigma}^{-2},
\end{align*}
and $l$ is a solution of
\begin{align*}
	y'' = y' + 2\sigma^{-2}
\end{align*}
(see \cite{PZ}). The assertion follows from Theorem \ref{genthe} for $V\equiv 0$.
\end{example}
Next we present an example  illustrating the method for finding Laplace transform of a solution of another hyperbolic PDE of the second order - a solution of wave equation with axial symmetry. We would need the following lemma.
\begin{lemma} \label{FKBES} Let $X$ be a Bessel process of nonnegative dimension starting from $\lambda\geq 0$ and let $A$ be its generator. Suppose that for a  nonnegative bounded or uniformly H$\ddot{o}$lder-continuous function $V$ and a function $f$ of polynomial growth there is an unique solution of polynomial growth of the Cauchy problem
$$
	\frac{\partial w}{\partial t} = Aw - Vw, \quad w(0,\lambda) = f(\lambda).
$$
Then the pair $(\sigma, \mu)$, where $\sigma \equiv 1, \mu(x) = \frac{\delta -1}{2x}$,  is in $\mathcal{FK}$-class.
\end{lemma}
\begin{proof}
We use the observation from the proof of Proposition 2.14 that it is enough to prove \eqref{sup0}, i.e. that
for all $s\leq t\leq T$, $m\geq 1$
\begin{equation}\label{sup}
	\mathbb{E}\sup_{u\in[s,t]}X_u^{2m} \leq C (1+\lambda^{2m})e^{C(t-s)},
\end{equation}
where $C = C(m,T)$ is a constant and $X_0 = \lambda$.
Let us observe that there exists $M\geq 1$ such that $X_u^2 \leq \sum_{i=1}^{M-1}(B_u^{(i)})^2 + (\lambda + B_u^{(M)})^2$, where $B^{(i)}$ are independent standard Brownian motions. In result there exists a constant $C = C(m,M)$ such that
\begin{align*}
	\mathbb{E}\sup_{u\in[s,t]}X_u^{2m} &\leq  C \Big( \lambda^{2m} + \mathbb{E}\sup_{u\in[s,t]}B_u^{2m}\Big)
	\leq C(\lambda^{2m} + (t-s)^m)\\
	 &\leq C^* (\lambda^{2m}+1)e^{C(t-s)},
\end{align*}
where $C^* = C^*(m,M,T)$.
\end{proof}
\begin{example}(Wave equation with axial symmetry) Consider, for $w = w(\lambda, x)$, the following PDE
\begin{align}\label{WE}
	\frac{\partial^2 w}{\partial x^2} &= a^2\Big(\frac{\partial^2 w}{\partial\lambda^2}+\frac{1}{\lambda}\frac{\partial w}{\partial \lambda}\Big), \
		\\
		 w(\lambda,0) &= f(\lambda), \quad \frac{\partial w}{\partial x}(\lambda,0) = 0,\notag
\end{align}
where $x\geq 0, \lambda > 0$, $a>0$ and $w,f$ are in $\mathcal{C}^2$.
 The form of the Green function for problem (\ref{WE}) is known. 
  However,
it is in the form of complicated infinite series depending on zeros of an appropriate Bessel function - for details see \cite{P}. Using matching diffusion technique, under the assumptions of hypothesis HF, we 
can obtain quickly the Laplace transform of a
 solution of (\ref{WE}). \\
If a bounded solution of (\ref{WE}) satisfies
$$
\frac{\partial w}{\partial x}(\lambda,0) = 0, \quad w^+(0+,x) = 0,
$$
where $w^+$ denotes the right derivative with respect to the scale function (see for instance \cite{BS} Section 21 in Appendix 1), then the first condition  guarantees that  the function $w(\lambda, \cdot)$ belongs to the domain of reflected Brownian motion and the second that $w(\cdot, x)$ belongs to the domain of Bessel process.
We choose $\hat{E} = [0,\infty), E = (0,\infty)$ and $\hat{\sigma}\equiv \sqrt{2}, \hat{\mu}\equiv 0, \sigma\equiv a\sqrt{2}, \mu(\lambda) = \frac{a^2}{\lambda}$.
 Observe that $\theta_t := X_{\frac{t}{2a^2}}$ is a Bessel process with index $0$ and $\hat{X}_t = |\sqrt{2}B_t + x|$, where $X_0 = x$.  Having in mind Lemma \ref{FKBES} and assuming that for $\hat{X}$ a function $w$ satisfies the hypothesis HF  we can use now Theorem \ref{genthe} to obtain
\begin{equation}\label{sol}
	\mathbb{E}w(\lambda, |\sqrt{2}B_t + x|) = \mathbb{E}w(\theta_{2a^2t},x), \quad \theta_0 = \lambda,
\end{equation}
$x\geq 0, \lambda > 0, t\geq 0$. To ease the notation put $2a^2 = 1$. If $x = 0$, then from (\ref{sol}) we have
\begin{align*}
 \mathbb{E}w(\lambda, |\sqrt{2}B_t|) = \mathbb{E}w(\theta_{t},0) = \mathbb{E}f(\theta_{t}).
\end{align*}
Hence, using the densities of $|B_t|$ and $\theta_t$, we obtain
\begin{align}\label{trsol}
	\int_0^{\infty}w(\lambda, \sqrt{2}z)\frac{\sqrt{2}}{\sqrt{\pi t}}e^{-\frac{z^2}{2t}}dz = \int_0^{\infty}f(z)z\frac{1}{t}\exp{\Big(-\frac{\lambda^2+z^2}{2t}\Big)}I_0(\lambda z/t)dz,
\end{align}
where $I_0$ is modified Bessel function.
This allows to give, using elementary algebra, the Laplace transform of $w(\lambda, \cdot)$
\begin{align*}
	\int_0^{\infty}e^{-\frac{z^2}{4t}}w(\lambda,z)dz = \sqrt{\pi} \int_0^{\infty}f(z)z\frac{1}{\sqrt{t}}\exp{\Big(-\frac{\lambda^2+z^2}{2t}\Big)}I_0(\lambda z/t)dz.
\end{align*}
\end{example}

At the end of the paper, we present some new results for a Jacobi diffusion that can be obtained using the matching diffusions technique. The formulae  are very simple in comparison with very complicated formulae obtained for the similar functionals
of Jacobi diffusion by Hurd and Kuznetsov \cite{HK}. Let us recall that a Jacobi diffusion takes values in $[0,1]$ and is given by
\begin{align*}
	dX_t = \sqrt{X_t(1-X_t)}dB_t + (b-aX_t)dt,\quad X_0 = \lambda\in [0,1],
\end{align*}
where $a,b\in\mathbb{R}$. The distribution of the process is complicated and only few closed formulae for functionals of Jacobi diffusion has been obtained so far (see \cite{HK} and \cite{BSII}).
\begin{proposition} Let $X$ be a Jacobi diffusion. If $t\geq 0$, $b\geq \frac12$, $a\geq b+\frac12$, $\alpha \geq 1 + 2b$ and $\gamma \geq (2(b-a)+1)^+$, then
\begin{align}
\label{up} \mathbb{E}X_t^{\alpha}e^{-\alpha(\frac{\alpha-1}{2}-b)\int_0^tX_u^{-1}du}&= \lambda^{\alpha}e^{\alpha(a-\alpha+1)t/2},\\
\label{down}\mathbb{E}(1-X_t)^{\gamma}e^{-\gamma(\frac{\gamma-1}{2}-b+a)\int_0^tX_u(1-X_u)^{-1}du}&= (1-\lambda)^{\gamma}e^{-\gamma tb}.
\end{align}
Moreover, (\ref{down}) holds also for $b<1/2$, $a>2$ and $\gamma \geq (2(b-a)+1)^+$.
\end{proposition}
\begin{proof} If $b\geq \frac12$ and $a\geq b+\frac12$, then both points $0$ and $1$ are unattainable for $X$ (see \cite[Section 4]{HK}). If $b<1/2$, $a>2$, then the point $1$ is the entrance-not-exit point (see  \cite[Section 4]{BSII}). To obtain (\ref{up}) we take
$\hat{X}_t = xe^{\alpha(a-\alpha+1)\frac{t}{2\beta}}$, where $x>0$, $\beta>1$, and use Theorem \ref{genthe} for $\hat{X},X$, $V(\lambda) = \frac{\alpha}{\lambda}(\frac{\alpha-1}{2}-b)$ and $f(\lambda,x) = \lambda^{\alpha}x^{\beta}$, for which the hypothesis HF is clearly satisfied. To obtain (\ref{down}) we take $\hat{X}_t = xe^{-b\alpha\frac{t}{2\beta}}$, $V(\lambda) = \frac{\alpha\lambda}{1-\lambda}(\frac{\alpha-1}{2}+a-b)$ and we use again Theorem \ref{genthe}.

\end{proof}
\bibliographystyle{plain}

\end{document}